\documentclass[10pt,leqno, english]{amsart}
\usepackage[colorlinks=true]{hyperref}
\usepackage{srcltx}
\usepackage[english]{babel}
\usepackage{amsmath, amsthm}
\numberwithin{equation}{section}
\usepackage{amsfonts}
\usepackage{amssymb}
\usepackage{bbm}
\usepackage{MnSymbol}

\newtheorem{proposition}{Proposition}[section]
\newtheorem{theorem}[proposition]{Theorem}
\newtheorem{lemma}[proposition]{Lemma}

\newtheorem{conjecture}[proposition]{Conjecture}

\newtheorem{claim}[proposition]{Claim}
\newtheorem*{problem*}{Inverse Sieve Problem}

\theoremstyle{definition}
\newtheorem{definition}[proposition]{Definition}

\theoremstyle{remark}
\newtheorem{remark}[proposition]{Remark}

\newtheorem*{remark*}{Remark}
\newtheorem*{remarks*}{Remarks}

\begin{document}

\title[The inverse sieve problem in global fields]{The inverse sieve problem for algebraic varieties over global fields}

\author[J. M. Menconi, M. Paredes, R. Sasyk]{Juan Manuel Menconi $^{1,2}$, Marcelo Paredes $^{2}$ \MakeLowercase{and} Rom\'an Sasyk $^{1,2}$}

\address{$^{1}$Instituto Argentino de Matem\'aticas Alberto P. Calder\'on-CONICET,
Saavedra 15, Piso 3 (1083), Buenos Aires, Argentina;}

\address{$^{2}$Departamento de Matem\'atica, Facultad de Ciencias Exactas y Naturales, Universidad de Buenos Aires, Argentina.}

\email{\textcolor[rgb]{0.00,0.00,0.84}{jmenconi@dm.uba.ar}}
\email{\textcolor[rgb]{0.00,0.00,0.84}{mparedes@dm.uba.ar}}
\email{\textcolor[rgb]{0.00,0.00,0.84}{rsasyk@dm.uba.ar}}

\subjclass[14Q20]{11G35, 11G50, 11N35, 11N69}


\begin{abstract}	
Let $K$ be a global field and let $Z$ be a geometrically irreducible algebraic variety defined over $K$. We show that if a big set $S\subseteq Z$ of rational points of bounded height occupies few residue classes modulo $\mathfrak{p}$ for many prime ideals $\mathfrak{p}$, then a positive proportion of $S$ must lie in the zero set of a polynomial of low degree that does not vanish at $Z$. 
This generalizes a result of Walsh who studied the case when $S\subseteq \{0,\ldots ,N\}^{d}$.
\end{abstract}

\maketitle

\section{Introduction}

Let $S$ be a random set of integers. In arithmetic combinatorics, it is usual to establish ``inverse theorems'', in the sense that if $S$ posses some specific arithmetic property, then $S$ belongs to a certain family of subsets of integers; hence providing a classification for such kind of $S$. For the purpose of this article, the arithmetic property in question to be studied is the equidistribution of the set $S$. Here, by an equidistributed subset of integers $S$ we mean that $S$ is well-distributed modulo $p$ for many primes $p$ (note that this is weaker than being well-distributed modulo $m$ for many moduli $m$).  We expect that a random set $S$ is fairly well-distributed. Thus, an ``inverse problem'' here would be to understand whether a set that occupies few residue classes modulo $p$ for many primes $p$ has some specific structure. In this generality, this has been stated as follows.

\begin{problem*}[see \cite{Croot, Helfgott}]
Suppose that a set $S\subseteq \{0,\ldots ,N\}^{d}$ occupies very few residue classes modulo $p$ for many primes $p$. Then, either $S$ is small, or it possesses some strong algebraic structure.
\end{problem*}

In order to give a concrete example that motivated this sort of problem, consider a subset $S\subseteq \{0,\ldots ,N\}$ satisfying that $S_{p}:=\{x\ (\text{mod}\ p):x\in S\}$ has at most $\alpha p$ elements for many primes in the interval $[1,N]$, with $0<\alpha<1$. Gallagher's sieve (Theorem 1 in \cite{Gallagher}) implies that $|S|\leq c(\alpha)N^{\alpha}$. Let us further suppose that $|S_{p}|\leq \frac{p-1}{2}$ for all $p\leq N^{\frac{1}{2}}$. The large sieve implies (see \cite{Mo})
\begin{equation}
|S|\leq CN^{\frac{1}{2}},
\label{large sieve}
\end{equation}
where $C$ is an absolute constant. The bound \eqref{large sieve} is essentially sharp, since if we consider $S$ to be the set of the squares lying in $\{0,\ldots ,N\}$, we see that $S$ occupies at most $\frac{1}{2}(p-1)$ residue classes for all primes $p$ and $|S|\sim N^{\frac{1}{2}}$. More generally, if $S$ is the image of a quadratic polynomial  $aX^{2}+bX+c\in \mathbb{Z}[X]$, we also have that $S$ occupies at most $\frac{p-1}{2}$ residue classes for all primes $p$ not dividing $a$ and $|S|\sim N^{\frac{1}{2}}$. Thus, we may ask if there are any other examples for which the bound \eqref{large sieve} is almost optimal. This discussion on the large sieve together with the example of the squares, led Helfgott and Venkatesh \cite{Helfgott} and independently Croot and Elsholtz \cite{Croot} to conjecture that any badly distributed set $S$ of size close to $N^{\frac{1}{2}}$ should ``essentially'' be the image of a quadratic polynomial. More precisely, they posed the following conjecture.

\begin{conjecture}[Problem 7.4 in \cite{Croot}, Guess in \cite{Helfgott}]
Let $S\subseteq \{0,\ldots ,N\}$ of size $|S|\geq N^{\varepsilon}$ occupying less than $\alpha p$ residue classes for some $0<\alpha<1$ and every prime $p$. Then all but $O(N^{o(1)})$ elements of $S$ are contained in the set of values of a polynomial $f\in \mathbb{Z}[X]$ with coefficients and degree bounded in terms of $\alpha$ and $\varepsilon$.
\label{inverse sieve0}
\end{conjecture}

We remark that  Conjecture \ref{inverse sieve0} is known as the inverse problem for the large sieve (see Conjecture 1.4 in \cite{Green}). The name is due to the fact that Conjecture \ref{inverse sieve0} classifies all sets of integers $S$ that are obtained after sieving $\frac{p-1}{2}$ residue classes modulo $p$ over the primes $p\leq N^{\frac{1}{2}}$, and that have size close to $N^{\frac{1}{2}}$. In Section 4.2 of \cite{Helfgott} Helfgott and Venkatesh remarked that Conjecture \ref{inverse sieve0} implies that there are $O(N^{\varepsilon})$ points on an irrational curve, which is considered a very hard problem.  In the same article \cite{Helfgott}, the authors proved the following ``higher dimensional'' variant of Conjecture \ref{inverse sieve0}.

\begin{theorem}[Theorem 1.1 in {\cite{Helfgott}}]
Let $\varepsilon,\alpha>0$. There exist constants $c_{1}=c_{1}(\alpha,\varepsilon)$ and $c_{2}=c_{2}(\alpha,\varepsilon)$ such that the following holds. Let $S\subseteq \{0,\ldots ,N\}^{2}$ be a subset such that the number of residues $\{(x,y)\ (\text{mod}\ p):(x,y)\in S\}$ is at most $\alpha p$ for every prime $p$. Then, at least one of the following holds:
\begin{itemize}
\item $|S|\leq c_{1}N^{\varepsilon}$, or
\item there exists a non-zero polynomial of degree $c_{2}$ in $\mathbb{Z}[X,Y]$, vanishising in at least $(1-\varepsilon)|S|$ points of $S$.
\label{Helf} 
\end{itemize}
\end{theorem}

The argument of the proof of Theorem \ref{Helf} is based on the larger sieve of Gallagher. Specifically, they adapted the determinant method of Bombieri-Pila (see \cite{BP}) to give a two dimensional version of Gallagher's sieve. Furthermore their methods gave another proof of the estimates of Bombieri-Pila for plane curves (see Theorem 5 in \cite{BP}). While the proof of Theorem \ref{Helf} is closely linked to the original proof in \cite{BP}, it is remarkable that the method in \cite{Helfgott} uses ``local data'', i.e. the size of the residue classes of the points on a curve, instead of analytic data as in \cite{BP}, which is useful in other contexts (see, \cite{Sedunova}, where an analogue of Theorem 5 in \cite{BP} is obtained for function fields of genus $0$).
 
Helfgott and Venkatesh conjectured that a similar result to Theorem \ref{Helf} should hold for subsets of integers lying in $\mathbb{Z}^{d}$ for $d\geq 3$. Their methods, however, seem to handle only the case when $S$ occupies very few residue classes, specifically at most $\alpha p$ residue classes for all primes $p$. Note that these sets are not what we would expect for a general badly distributed set, where the number of residues classes can be at most $O(p^{d-1})$. Using a subtle inductive argument, together with the larger sieve and the polynomial method, in \cite{Walsh2}, Walsh solved this conjecture by proving the following theorem.

\begin{theorem}[Theorem 1.1 in {\cite{Walsh2}}]
Let $0\leq k<d$ be integers and let $\varepsilon,\alpha,\eta>0$ be positive real numbers. Then, there exists a constant $C$ depending only on the above parameters, such that for any set $S\subseteq \{0,\ldots ,N\}^{d}$ occupying less than $\alpha p^{k}$ residue classes for every prime $p$, at least one of the following holds:
\begin{itemize}
\item ($S$ is small) $|S|\lesssim_{d,k,\varepsilon,\alpha} N^{k-1+\varepsilon}$;
\item ($S$ is strongly algebraic) there exists a non-zero polynomial of degree at most $C$ in $\mathbb{Z}[X_{1},\ldots ,X_{d}]$ with coefficients bounded by $N^{C}$ vanishing at more than $(1-\eta)|S|$ points of $S$.
\end{itemize}
\label{basis}
\end{theorem}

Let us emphasize the important point that Theorem \ref{basis} means that there exist constants $c$, $C$ depending on the parameters $d,k,\varepsilon,\alpha,\eta$ such that for any set $S$ satisfying the hypothesis of the theorem, we have $|S|\leq cN^{k-1+\varepsilon}$ or there exists a non-zero polynomial $f\in \mathbb{Z}[X_{1},\ldots ,X_{d}]$ of degree at most $C$ and coefficients bounded by $N^{C}$ vanishing at more than $(1-\eta)|S|$ points of $S$.

In this article we are interested in investigating the Inverse Sieve Problem in the context of global fields. To that end, let $K$  be a global field and denote by $\mathcal{O}_{K}$ its ring of integers. A natural generalization to global fields of subsets  lying in $\{0,\ldots N\}^{d}$ is to consider the subsets $S\subseteq \{{\bf x}\in \mathcal{O}_{K}^{d}:H({\bf x})\leq N\}$ where $H$ is a height function, in the sense of diophantine geometry. If such a subset $S$ is small, or if it lies in the zero set of $\mathcal{O}_{K}$-points of an affine variety $Z\subseteq \mathbb{A}^{d}$ of dimension $l<d$ defined over $K$, then by classical bounds we have that for all primes $\mathfrak{p}$ in $K$, $Z(\mathcal{O}_{K}/\mathfrak{p})$ has at most $\lesssim_{Z}|\mathcal{O}_{K}/\mathfrak{p}|^{l}$ points. Thus, one may ask if a similar principle as in the Inverse Sieve Problem holds in the context of global fields. In this article we adapt the proof of Walsh to show that this is indeed the case. More precisely, we prove the following result.

\begin{theorem}
Let $0\leq k<d$ be integers and let $\varepsilon, \alpha,\eta>0$ be positive real numbers. Let $K$ be a global field of degree $d_{K}$. For a given $x\in K$ let $H(x)$ be the absolute multiplicative height of $x$. Then there exists a constant $C=C(d,k,\varepsilon,\alpha,\eta,K)$ such that for any set $S\subseteq \{x\in \mathcal{O}_{K}:H(x)\leq N\}^{d}$ occupying less than $\alpha|\mathcal{O}_{K}/\mathfrak{p}|^{k}$ residue classes for every prime $\mathfrak{p}$, at least one of the following holds:
\begin{itemize}
\item ($S$ is small) $|S|\lesssim_{d,k,\varepsilon,\alpha,K}N^{d_{K}(k-1)+\varepsilon}$;
\item ($S$ is strongly algebraic) there exists a non-zero polynomial of degree at most $C$ in $\mathcal{O}_{K}[X_{1},\ldots ,X_{d}]$ with coefficients of height bounded by $N^{C}$ vanishing at more than $(1-\eta)|S|$ points of $S$.
\end{itemize} 
\label{basis1}
\end{theorem} 
The reason why the exponent $d_{K}$ appears in the first case in Theorem \ref{basis1} is because we are counting $\mathcal{O}_{K}$-points with $K$ a possibly non-trivial extension of $\mathbb{Q}$ and one expects to have a $d_{K}$-power of the usual quantities. For instance, the line $x=y$ has $\sim_{K,\varepsilon}N^{d_{K}+\varepsilon}$ $\mathcal{O}_{K}$-points of absolute height at most $N$. 

In some situations, it is possible to have additional information about $S$, for instance, that $S$ already lies in an affine variety $Z$ defined over $K$, say geometrically irreducible. In this case, the statement of Theorem \ref{basis1} is trivial, since its second condition already holds. However, it may happen that $S$ occupies even fewer residue classes than $Z$. In this case, we can prove a sharper result than Theorem \ref{basis1}, as it can be seen in the next theorem.

\begin{theorem}
Let $0\leq k<d$ be integers, $D,M>0$ a positive integer, and let $\varepsilon,\alpha,\eta>0$ be positive real numbers. Let $K$ be a global field of degree $d_{K}$. For a given $x\in K$ let $H(x)$ be the absolute multiplicative height of $x$. Then there exists a constant $C=C(d,k,\varepsilon,\alpha,\eta,K,D,M)$  such that for any set $S\subseteq \{x\in \mathcal{O}_{K}:H(x)\leq N\}^{M}$  occupying less than $\alpha |\mathcal{O}_{K}/\mathfrak{p}|^{k}$ residue classes  for every prime $\mathfrak{p}$, and that lies in an affine variety $Z\subseteq \mathbb{A}^{M}$ defined over $K$, geometrically irreducible of dimension $d$ and degree $D$, at least one of the following holds:
\begin{itemize}
\item ($S$ is small) $|S|\lesssim_{d,k,\varepsilon,\alpha,K,D,M}N^{d_{K}(k-1)+\varepsilon}$;
\item ($S$ is strongly algebraic) there exists a polynomial  of degree at most $C$ and coefficients of height bounded by $N^C$
 in $\mathcal{O}_{K}[X_{1},\ldots ,X_{M}]$ vanishing at more than $(1-\eta)|S|$ points of $S$, that does not vanish at $Z$. 
\label{basis2}
\end{itemize}
\end{theorem}
As in \cite{Walsh2}, it can be shown that Theorem \ref{basis2} is sharp, and that $\varepsilon$ can not be taken to be equal to zero. 

In diophantine applications, one is usually interested in the rational points of some algebraic variety. Thus, one may ask if Theorem \ref{basis2} admits a ``projective version''. Here we also prove such version. Furthermore, as in \cite{Walsh2}, we only require that our sets occupy few residue classes for a ``dense'' subset of small primes. All this is summarized in the following theorem.

\begin{theorem}
Let $0\leq k<d-1$ be integers, $D,M>0$ positive integers, and let $\varepsilon,\alpha,\kappa,\eta>0$ be positive real numbers. Let $K$ be a global field of degree $d_{K}$ and $\mathcal{O}_{K}$ be its ring of integers. Set $Q=N^{\frac{\varepsilon}{2(d+1)}}$ and let 
$$P\subseteq\mathcal{P}(Q):=\left \{\mathfrak{p}\text{ prime of }\mathcal{O}_{K}:\left|\mathcal{O}_{K}/\mathfrak{p}\right|\leq Q\right \}$$
be a subset of primes satisfying 
$$w(P):=\sum_{\mathfrak{p}\in P}\dfrac{\log(\left|\mathcal{O}_{K}/\mathfrak{p}\right|)}{\left|\mathcal{O}_{K}/\mathfrak{p}\right|}\geq \kappa w(\mathcal{P}(Q)).$$
For a given ${\bf x}\in \mathbb{P}^{M}(K)$ let $H({\bf x})$ be its absolute multiplicative height. Then there exists a constant $C=C(d,k,\varepsilon,\alpha,\eta,K,D,M)$  such that for any set $S\subseteq \{{\bf x}\in \mathbb{P}^{M}(K):H({\bf x})\leq N\}$ occupying less than $\alpha |\mathcal{O}_{K}/\mathfrak{p}|^{k}$ residue classes for every prime $\mathfrak{p}$ (i.e. the image of $S$ in $\mathbb{P}^{M}(\mathcal{O}_{K}/\mathfrak{p})$ has at most $\alpha |\mathcal{O}_{K}/\mathfrak{p}|^{k}$ elements), and that lies in a projective variety $Z\subseteq \mathbb{P}^{M}$ defined over $K$, geometrically irreducible of dimension $d$ and degree $D$, at least one of the following holds:
\begin{itemize}
\item ($S$ is small) $|S|\lesssim_{d,k,\varepsilon,\alpha,K,D,M}N^{d_{K}k+\varepsilon}$;
\item ($S$ is strongly algebraic) there exists an homogeneous polynomial of degree at most $C$ 
and coefficients of height bounded by $N^C$
in $\mathcal{O}_{K}[X_{0},\ldots ,X_{M}]$ vanishing at more than $(1-\eta)|S|$ points of $S$, that does not vanish at $Z$. 
\label{basis3}
\end{itemize}
\end{theorem}
 It would be interesting to know if Theorem \ref{basis3} remains valid in the extreme case $k=d-1$. We defer a discussion about this case to the end of section  $\S$\ref{section 4}.

The proofs that we present here follow the general strategy developed by Walsh in \cite{Walsh2}. However, given the nature of the statements of our theorems, several new difficulties arise. First we need to adapt two kinds of estimates over $\mathbb{Z}$ to the corresponding estimates over global fields: those concerning the behavior of heights and those concerning the distribution of primes. To that end, in section $\S$\ref{section 2} we start by defining the height function that will be used throughout this article. While the theory of heights on number fields is very well documented, arguably this is not so in the case of function fields. Specifically, here we prove two statements for heights in function fields that we could not find in the literature. These are Proposition \ref{serre}, that states that points in $\mathbb{P}^{n}$ of height $\lesssim 1$ are lifted to points in $\mathbb{A}^{n+1}$ of height $\lesssim_{K} 1$, and Proposition \ref{S-integers bound}, which gives an upper bound for the number of points of bounded height in the group of $S$-units in a function field. Both results are well known for number fields (see, for instance, Section 13.4 in Chapter 13 of \cite{Serre} and Section 3 of \cite{Lang}, respectively). Concerning the distribution of primes, in section $\S$\ref{new section 3}, after recalling Landau Prime Ideal Theorem for number fields and the Riemann hypothesis over function fields, we extend the larger sieve of Gallagher, as it was presented in \cite{Walsh2}, to global fields. 

The second kind of difficulty is that we work with sets lying in algebraic varieties, and that the bounds in our theorems are uniform in the degree and dimension of such algebraic varieties. We overcome it in section $\S$\ref{section 3} in essentially two steps. First we use a standard argument  to reduce Theorem \ref{basis3} to a statement concerning  affine varieties. Then, we use Noether's normalization theorem to make a change of variables and reduce Theorem \ref{basis3} to a statement concerning badly distributed sets in an affine space. Because of the uniformity in our bounds, we need to have a nice control in the change of variables. Thus, we are led to prove in Theorem \ref{noether} an effective Noether's normalization theorem, which may be of interest in its own right. 

Finally, in section $\S$\ref{section 4} we prove Theorem \ref{basis3}. Unlike in the paper of Walsh \cite{Walsh2}, we follow the dependence of the parameters in the proofs, which brings the last technical difficulty of this article. We believe that having made explicit the dependence of the constants may be useful for some diophantine applications.

\section{Heights in global fields}
\label{section 2}

The purpose of this section is twofold. First we establish a normalization of the absolute values of a global field. We use this to define the height function that will be used in this article and recall some basic properties of it. Secondly, we state two propositions concerning estimates for the number of points in the group of $S$-units of a global field, which are well known for number fields, but for which we could not find a reference for function fields. The presentation in this section has been influenced by the standard references \cite{Bombieri, Hindry, Lang, Serre}.

\subsection{Absolute values and relative height}
 Throughout this paper, $K$ denotes a global field, i.e. a finite separable extension of $\mathbb Q$ or $\mathbb{F}_{q}(T)$, in  which case we further assume that the field of constants is $\mathbb{F}_{q}$. We will denote by $d_{K}$ the degree of the extension  $K/ \mathbbm{k}$,  where $\mathbbm{k}$ indistinctively denotes the base fields  $\mathbb{Q}$  or  $\mathbb{F}_{q}(T)$.

Let $K$ be a number field and $\mathcal{O}_{K}$ its ring of integers. Then each embedding $\sigma:K\hookrightarrow \mathbb{C}$ induces a place $v$, by means of the equation
$$||x||_{v}:=|\sigma(x)|^{\frac{n_{v}}{d_{K}}}_{\infty},$$
where $|\cdot |_{\infty}$ denotes the absolute value of $\mathbb{R}$ or $\mathbb{C}$ and $n_{v}=1$ or $2$ respectively. 
Such places will be called the places at infinity, and denoted by $M_{K,\infty}$. Note that $\sum_{v\in M_{K,\infty}}n_{v}=d_{K}$. They are all the archimedean places of $K$. Since the complex embeddings come in pairs that differ by complex conjugation, we have $|M_{K,\infty}|\leq d_{K}$. 

Now let $\mathfrak{p}$ be a non-zero prime ideal of the number field $K$, and denote by $\text{ord}_{\mathfrak{p}}$ the usual $\mathfrak{p}$-adic valuation. Associated to $\mathfrak{p}$, we have the place $v$ in $K$ given by the equation
$$||x||_{v}:=|x|_{\mathfrak{p}}:=\mathcal{N}_{K}(\mathfrak{p})^{-\frac{\text{ord}_{\mathfrak{p}}(x)}{d_{K}}},$$
where $\mathcal{N}_{K}(\mathfrak{p})$ denotes the cardinal of the finite quotient $\mathcal{O}_{K}/\mathfrak{p}$. We will also denote $\mathcal{O}_{\mathfrak{p}}$ for the localization at $\mathfrak{p}$ of the ring $\mathcal{O}_{K}$. Such places will be called the finite places, and denoted by $M_{K,\text{fin}}$. They are all the non-archimedean places of $K$. 
The set of places of $K$ is then the union $M_{K,\infty}\cup M_{K,\text{fin}}$, and we denote it by $M_{K}$. For any finite subset $S\subseteq M_{K}$ containing the infinite places $M_{K,\infty}$, we define the ring of $S$-integers of $K$ to be the set
$$\mathcal{O}_{K,S}:=\left\{ x\in K: ||x||_{v}\leq 1\text{ for all }v\in M_{K},v\notin S\right\}.$$
Observe that $\mathcal{O}_{K}=\mathcal{O}_{K,S}$ for $S=M_{K,\infty}$.
The norm of a non-zero ideal $I\subseteq \mathcal{O}_{K,S}$, denoted by $\mathcal{N}_{K,S}(I)$, is just the cardinal of the finite quotient $\mathcal{O}_{K,S}/I$. The prime ideals of $\mathcal{O}_{K,S}$ correspond to the prime ideals $\mathfrak{p}\mathcal{O}_{K,S}$ where $\mathfrak{p}$ is a prime ideal of $\mathcal{O}_{K}$ not lying in $S$.

Now, let us suppose that $K$ is a function field over $\mathbb{F}_{q}$, such that $\mathbb{F}_{q}$ is algebraically closed in $K$ (in other words, the constant field of $K$ is $\mathbb{F}_{q}$). A prime in $K$ is, by definition, a discrete valuation ring $\mathcal{O}$ with maximal ideal $\mathfrak{p}$ such that $\mathbb{F}_{q}\subseteq \mathfrak{p}$ and the quotient field of $\mathcal{O}$ equal to $K$. By abuse of notation, when we refer to a prime in $K$, we will refer to the maximal ideal $\mathfrak{p}$. We will also denote $\mathcal{O}_{\mathfrak{p}}$ to the corresponding discrete valuation ring. Associated to $\mathfrak{p}$, we have the usual $\mathfrak{p}$-adic valuation, that we will denote by $\text{ord}_{\mathfrak{p}}$. The degree of $\mathfrak{p}$, denoted by $\deg(\mathfrak{p})$ will be the dimension of $\mathcal{O}_{\mathfrak{p}}/\mathfrak{p}$ as an $\mathbb{F}_{q}$-vectorial space, which is finite. Then the norm of $\mathfrak{p}$ is defined as $\mathcal{N}_{K}(\mathfrak{p}):=q^{\deg(\mathfrak{p})}$. 
Any prime $\mathfrak{p}$ of $K$ induces a place $v$ in $K$ by the equation
$$||x||_{v}:=|x|_{\mathfrak{p}}:=\mathcal{N}_{K}(\mathfrak{p})^{-\frac{\text{ord}_{\mathfrak{p}}(x)}{d_{K}}}.$$
They are all the places in $K$. The set of all places in $K$ is denoted by $M_{K}$. As in the case of number fields, for any nonempty finite subset $S\subseteq M_{K}$, we define the ring of $S$-integers of $K$ to be the set
$$\mathcal{O}_{K,S}:=\left\{ x\in K: ||x||_{v}\leq 1\text{ for all }v\in M_{K},v\notin S\right\}.$$
Given $x\in \mathcal{O}_{K,S}$ we define $\mathcal{N}_{K,S}(x):=\prod_{\mathfrak{p}\notin S}\mathcal{N}_{K}(\mathfrak{p})^{\text{ord}_{\mathfrak{p}}(x)}$. By definition, $\text{ord}_{\mathfrak{p}}(x)\geq 0$ for all $\mathfrak{p}\notin S$, so that $\mathcal{N}_{K,S}(x)$ is a positive integer. 
A prime in $\mathcal{O}_{K,S}$ will be any prime $\mathfrak{p}\in K$ not in $S$. When $S=\{v\}$, we will usually denote $\mathcal{O}_{K,S}=\mathcal{O}_{K}$. If $w\in M_{\mathbbm{k}}$ is the place below $v$, we will denote $M_{K,\infty}:=\{v'\in M_{K}:v'|w\}$. Note that $|M_{K,\infty}|\leq d_{K}$. 

Now, given a global field $K$, we define the absolute multiplicative projective height of $K$ of a point ${\bf x}=(x_{0}:\ldots :x_{n})\in \mathbb{P}^{n}(K)$, to be the function
$$H({\bf x}):=\displaystyle \prod_{v\in M_{K}}\max_{i}\{||x_{i}||_{v}\},$$
and the relative multiplicative projective height by
$$H_{K}({\bf x}):=H({\bf x})^{d_{K}}.$$
If $x\in K$, $H_{K}(x)$ will always denote the projective height $H_{K}(1:x)$. The next inequalities follow immediately from the definition of the height
\begin{equation}
H_{K}(x\cdot y)\leq H_{K}(x)\cdot H_{K}(y),
\label{height1}
\end{equation}
\begin{equation}
H_{K}(x+y)\leq 2^{d_{K}}H_{K}(x)H_{K}(y).
\label{height2}
\end{equation}
Also, from the product formula it follows that for all $x\in K^{\ast}$,
\begin{equation}
H_{K}(x)=H_{K}(x^{-1}).
\label{height3}
\end{equation}
For our purposes, it will be necessary to understand how the affine height of a point behaves under the action of a polynomial. It is easy to show (see Proposition B.2.5. (a) in \cite{Hindry}) that if $P(T_{1},\ldots ,T_{n})=\sum_{(i_{1},\ldots ,i_{n})}c_{i_{1},\ldots ,i_{n}}T_{1}^{i_{1}}\cdots T_{n}^{i_{n}}$, ${\bf c}=(c_{i_{1},\ldots ,i_{n}})_{i_{1},\ldots ,i_{n}}$ and $R$ is the number of $(i_{1},\ldots, i_{n})$ with $c_{i_{1},\ldots ,i_{n}}\neq 0$, we have
\begin{equation}
H_{K}(P({\bf x}))\leq R^{d_{K}}H_{K}(1:{\bf c})H_{K}(1:{\bf x})^{\deg(P)}.
\label{polynomialheight2}
\end{equation}
Given a set of places $S$ and ${\bf x}=(x_{1},\ldots, x_{n})\in \mathcal{O}_{K,S}^{n}$, we have the bound
\begin{equation}
H_{K}(x_{1}:\ldots :x_{n})\leq H_{K}(1:x_{1}:\ldots :x_{n})\leq \max_{i}\{H_{K}(x_{i})\}^{|S|}.
\label{height1}
\end{equation}
Also, for any $x\in \mathcal{O}_{K,S}\backslash \{0\}$, it holds
\begin{equation}
\mathcal{N}_{K}(x)\leq H_{K}(x).
\label{norm}
\end{equation}

In section $\S$\ref{section 3} we will require to lift a bounded set in projective space to a set in affine space. The next proposition states that this can be done in a controlled manner.
\begin{proposition}\label{serre}
Let $K$ be a global field, let $S$ be a finite set of places, with the additional condition that $M_{K,\infty}\subseteq S$ if $K$ is a number field, and let $d\geq 1$ be an integer. There exists $c=c(K,S,d)$ such that for every ${\bf x}\in \mathbb{P}^{d}(K)$ there exists $(y_{0},\ldots ,y_{d})\in \mathcal{O}_{K,S}^{d+1}$ a lift of $\bf x$ such that
$$H_{K}(1:y_{0}:\ldots :y_{d})\leq c H_{K}(\bf{x}).$$
\end{proposition}
This is proved in Section 13.4 of \cite{Serre} when $K$ is a number field and $S=M_{K,\infty}$. For the sake of completeness, we include the proof of this more general statement in section $\S$\ref{appendix}.

In section $\S$\ref{section 4} we will need estimates of the numbers of points in $\mathcal{O}_{K}$  of a given height. This is addressed by the following proposition.
\begin{proposition}[$S$-integer points of bounded height]
Let $K$ be a global field, and let $S\subseteq M_{K}$ be a nonempty finite subset of places of $K$ (which we require that contains the infinite places when $K$ is a number field, and the place $v$ fixed to define $\mathcal{O}_{K}$ when $K$ is a function field). Then
$$\left|\left \{ x\in \mathcal{O}_{K,S}:H_{K}(x)\leq N\right \}\right| \leq c''(K) N(\log(N))^{|S|}.$$
\label{S-integers bound}
\end{proposition}
When $K$ is a number field, sharper estimates than Proposition \ref{S-integers bound} hold; for instance, see Theorem 5.2 in \cite{Lang} for the case $S=M_{K,\infty}$ and Theorem $1.1$ in \cite{Barroero} for arbitrary $S$. Moreover, Theorem 1.1 in \cite{Barroero} gives effective estimates.  Since we could not find a reference for Proposition \ref{S-integers bound} over function fields we provide a proof in the appendix, section $\S$\ref{appendix}, of this article.\\

In the remaining of the paper we will use the following notation,
\begin{equation}
[N]_{\mathcal{O}_{K}}^{n}:=\{{\bf x}=(x_{1},\ldots ,x_{n})\in \mathcal{O}_{K}^{n}:\max_{i}\{H_{K}(x_{i})\}\leq N\},
\end{equation}
\begin{equation}
[N]_{\mathbb{A}^{n}(\mathcal{O}_{K})}:=\{{\bf x}=(x_{1},\ldots ,x_{n})\in \mathcal{O}_{K}^{n}:H_{K}(1:x_{1}:\ldots :x_{n})\leq N\}.
\end{equation}
\begin{equation}
[N]_{\mathbb{P}^{n}(K)}:=\{{\bf x}=(x_{0}:\ldots :x_{n})\in \mathbb{P}^{n}(K):H_{K}({\bf x})\leq N\}.
\end{equation}
Note that since $\max_{i}\{H_{K}(x_{i})\}\leq H_{K}(1:x_{1}:\ldots :x_{n})$ for all ${\bf x}=(x_{1},\ldots ,x_{n})\in K^{n}$ we have
\begin{equation}
[N]_{\mathbb{A}^{n}(\mathcal{O}_{k})}\subseteq [N]_{\mathcal{O}_{K}}^{n}.
\label{obvio}
\end{equation}

\section{The larger sieve over global fields}
\label{new section 3}

In this section, we extend the larger sieve of Gallagher as it was presented in  \cite{Walsh2}, to global fields. To this end, we first recall some basic inequalities concerning the distribution of primes in global fields.

\subsection{Distribution of primes in global fields}

Let $K$ be a global field.
For any prime $\mathfrak{p}$ of $\mathcal{O}_{K}$ we have a reduction map $\pi_{\mathfrak{p}}:\mathbb{P}^{n}(K)\rightarrow \mathbb{P}^{n}(\mathcal{O}_{K}/\mathfrak{p})$. If ${\bf x},{\bf y}\in \mathbb{P}^{n}(K)$ by ${\bf x}\equiv {\bf y}(\text{mod}\ \mathfrak{p})$ we will mean $\pi_{\mathfrak{p}}({\bf x})=\pi_{\mathfrak{p}}({\bf y})$. Note that if $f\in \mathcal{O}_{K}[X_{0},\ldots ,X_{n}]$ is an homogeneous polynomial such that $f({\bf x})=0$, then ${\bf x}\equiv {\bf y}(\text{mod}\ \mathfrak{p})$ implies $f({\bf y})\equiv f({\bf x})\equiv 0 (\text{mod}\ \mathfrak{p})$. Likewise we have a reduction map $\pi_{\mathfrak{p}}:\mathcal{O}_{K}^{n}\rightarrow (\mathcal{O}_{K}/\mathfrak{p})^{n}$, and for ${\bf x},{\bf y}\in \mathcal{O}_{K}^{n}$ we will denote ${\bf x}\equiv {\bf y}(\text{mod}\ \mathfrak{p})$ if $\pi_{\mathfrak{p}}({\bf x})=\pi_{\mathfrak{p}}({\bf y})$.

If $S\subseteq [N]_{\mathbb{A}^{n}(\mathcal{O}_{K})},[N]_{\mathcal{O}_{K}}^{n}$, or $[N]_{\mathbb{P}^{n}(K)}$ and $\mathfrak{p}$ is a prime of $\mathcal{O}_{K}$ we will use the notation $[S]_{\mathfrak{p}}:=\pi_{\mathfrak{p}}(S)$ where $\pi_{\mathfrak{p}}$ is the corresponding reduction map. 
For any $Q>0$, let us denote: 
$$\mathcal{P}:=\{\mathfrak{p}\text{ prime in }\mathcal{O}_{K}\},$$
$$\mathcal{P}(Q):=\{\mathfrak{p}\in \mathcal{P}:\mathcal{N}_{K}(\mathfrak{p})\leq Q\}.$$
If $P\subseteq \mathcal{P}(Q)$, we denote 
$$w(P):=\sum_{\mathfrak{p}\in P}\frac{\log(\mathcal{N}_{K}(\mathfrak{p}))}{\mathcal{N}_{K}(\mathfrak{p})}.$$
If $K$ is a global field, there exist constants
$c_{1,K}, c_{2,K}, c_{3,K}$ and $c_{4,K}$ such that for all $Q>0$ it holds that 
\begin{equation}
c_{1,K}\log(Q)\leq w(\mathcal{P}(Q))\leq c_{2,K}\log(Q),
\label{landau1}
\end{equation}
and 
 \begin{equation}
c_{3,K}Q\leq \displaystyle \sum_{\mathfrak{p}\in \mathcal{P}(Q)}\log(\mathcal{N}_{K}(\mathfrak{p}))\leq c_{4,K}Q.
\label{landau2}
\end{equation}
Indeed, if $K$ is a number field, \eqref{landau1} and \eqref{landau2} follow from Landau Prime Ideal Theorem (see Theorem 5.33 in \cite{Iwaniec}). Meanwhile, if $K$ is a function field over $\mathbb{F}_{q}$ of genus $g$, this follows from the Riemann Hypothesis over function fields (see Theorem 5.12 in \cite{Rosen}). Note that in this case, the constants will also depend on the (degree of the) prime $v$ that we choose to define $\mathcal{O}_{K}$.

\subsection{Larger sieve over global fields}

Let $S\subseteq [N]\subseteq \mathbb Z$  and let $Q>0$ be a parameter. The larger sieve of Gallagher, Theorem 1 in \cite{Gallagher}, consists on counting in two different ways the number of pairs $x,y\in S$ and positive primes $p$ with $p\leq Q$ such that $x\equiv y(\text{mod}\ p)$. In our context, $S\subseteq [N]_{\mathcal{O}_{K}}$, and we want to count the number of pairs $x,y\in S$ and primes $\mathfrak{p}\in \mathcal{P}(Q)$ such that $x\equiv y(\text{mod}\ \mathfrak{p})$. Thus, we have
\begin{align}\label{1.1}
\displaystyle \sum_{\substack{x,y\in S\\ x\neq y}}\sum_{\mathfrak{p}\in \mathcal{P}(Q)}1_{x\equiv y(\text{mod}\ \mathfrak{p})}\log(\mathcal{N}_{K}(\mathfrak{p}))&=\sum_{\substack{x,y\in S\\ x\neq y}}\log\left( \prod_{\substack{\mathfrak{p}|x-y\\ \mathfrak{p}\in \mathcal{P}(Q)}}\mathcal{N}_{K}(\mathfrak{p}) \right)\\
&\leq \sum_{\substack{x,y\in S\\ x\neq y}}\log(\mathcal{N}_{K}(x-y)).\nonumber
\end{align}
Using \eqref{norm} and \eqref{height2} we obtain $\mathcal{N}_{K}(x-y)\leq 2^{d_{K}}H_{K}(x)H_{K}(y)$, that is smaller than $N^{3}$ if $N>2^{d_{K}}$. Hence
\begin{equation}
\sum_{\substack{x,y\in S\\ x\neq y}}\sum_{\mathfrak{p}\in \mathcal{P}(Q)}1_{x\equiv y(\text{mod}\ \mathfrak{p})}\log(\mathcal{N}_{K}(\mathfrak{p}))
\leq 3|S|^{2}\log(N).
\label{1.5}
\end{equation}
On the other hand, if $S(a,\mathfrak{p}):=\{x\in S: x\equiv a(\text{mod}\ \mathfrak{p})\}$, the left hand side of \eqref{1.5} is equal to
\begin{equation}
\displaystyle \sum_{\mathfrak{p}\in \mathcal{P}(Q)}\sum_{a(\text{mod}\ \mathfrak{p})}|S(a,\mathfrak{p})|^{2}\log(\mathcal{N}_{K}(\mathfrak{p}))-|S|\sum_{\mathfrak{p}\in \mathcal{P}(Q)}\log(\mathcal{N}_{K}(\mathfrak{p})).
\label{2.1}
\end{equation}
Thus, \eqref{1.5} and \eqref{2.1} imply
\begin{equation}
\displaystyle \sum_{\mathfrak{p}\in \mathcal{P}(Q)}\sum_{a(\text{mod}\ \mathfrak{p})}|S(a,\mathfrak{p})|^{2}\log(\mathcal{N}_{K}(\mathfrak{p}))-|S|\sum_{\mathfrak{p}\in \mathcal{P}(Q)}\log(\mathcal{N}_{K}(\mathfrak{p}))\leq 3|S|^{2}\log(N).
\label{larger sieve}
\end{equation}
Note that the above argument also works if $S\subseteq [N]_{\mathcal{O}_{K}}^{d}$. Indeed, if $\pi_{1}:K^{n}\rightarrow K$ denotes the projection on the first coordinate, then ${\bf x}\equiv {\bf y}(\text{mod}\ \mathfrak{p})$ implies $\mathfrak{p}|\pi_{1}({\bf x})-\pi_{1}({\bf y})$, thus the argument to prove \eqref{larger sieve} still holds.
\begin{lemma}[{Compare to Lemma 3.1 in \cite{Walsh2}}]
Let $X\subseteq [N]_{\mathcal{O}_{K}}$, $Q=N^{\gamma}$, $\gamma>0$. Let $\kappa,\mu$ be positive real numbers. Suppose that there is a set of primes $P\subseteq \mathcal{P}(Q)$ with $w(P)\geq \kappa w(\mathcal{P}(Q))$ such that for any prime $\mathfrak{p}\in P$ there are at least $\mu|X|$ elements of $X$ in at most $\alpha\mathcal{N}(\mathfrak{p})$ residue classes for some $\alpha>0$ independent of $\mathfrak{p}$. Then, there exists $C_{1}=C_{1}(\kappa,\mu,\gamma,K)$ such that if $\alpha\leq C_{1}$, it must be $|X|<Q$.
\label{lemma 3.1}
\end{lemma}

\begin{lemma}[{Compare to Lemma 3.2 in \cite{Walsh2}}]
Let $Q=N^{\gamma}$ for some $\gamma>0$ and let $P\subseteq \mathcal{P}(Q)$ be a set of primes with $w(P)\geq \kappa w(\mathcal{P}(Q))$ for some $\kappa>0$. Let $S\subseteq [N]_{\mathcal{O}_{K}}^{d}$ be a set occupying less than $\alpha$ residue classes modulo $\mathfrak{p}$ for every prime $\mathfrak{p}\in P$ and some constant $\alpha$, independent of $\mathfrak{p}$. Then there exists a constant $C_{2}=C_{2}(\alpha,\kappa,\gamma,K)$ such that  $|S|\leq C_{2}$.
\label{lemma 3.2}
\end{lemma}

Lemma \ref{lemma 3.1} and Lemma \ref{lemma 3.2} can be proved using \eqref{larger sieve}. The proofs are analogous to the corresponding proofs in \cite{Walsh2}, so we will not include them here. Also we remark that the constants $C_{1}$ and $C_{2}$ may be taken to be

\begin{equation}
C_{1}(\kappa,\mu,\gamma,K)=\dfrac{\kappa\mu^{2}\gamma}{c_{5,K}},\; c_{5,K}:=\frac{2(c_{4,K}+3)}{c_{1,K}},
\label{C1original}
\end{equation}

\begin{equation}
C_{2}(\alpha,\kappa,\gamma,K)=\max\left\{2\alpha,2\left( \dfrac{12\alpha }{c_{1,K}^{2}\gamma\kappa} \right)^{2\frac{c_{3,K}}{\gamma\kappa}} \right\},
\end{equation}
where $c_{1,K},\, c_{3,K}$ and  $c_{4,K}$ are the constants appearing in equations \eqref{landau1} and \eqref{landau2}.

\section{An effective change of variables}
\label{section 3}

The main result of this section is that Theorem \ref{basis3} follows from a similar result concerning affine spaces. In order to do that, we are going to make two reductions. The first step is to reduce our problem to that of studying subsets of affine varieties. We note that this is standard, for instance, see Chapter 14 in \cite{Serre}, where an upper bound for the number of rational points on a thin set lying in a projective space is obtained as a consequence of an analogous bound for thin sets lying in an affine space. The second step is to further reduce this problem by passing from an affine variety of dimension $d$  to the affine space $\mathbb{A}^{d}$. In order to do this, we are going to make a change of variables. This will be achieved by means of Noether's normalization. Since Theorem \ref{basis3} is uniform in the degree and dimension of the variety, we will require some effectiveness in the normalization. To that end, we are going to provide an effective version of the Noether's normalization theorem, with quite elementary methods, which is interesting in its own right.

\subsection{Reduction to the affine case}  In order to carry on the first step in the discussion above, we proceed to state the following variant of Theorem \ref{basis3} for affine varieties.

\begin{theorem}[Affine case of Theorem \ref{basis3}]
Let $0\leq k<d$ be integers, $D,M$ be positive integers, and let $\varepsilon, \alpha,\kappa,\eta>0$ be positive real numbers. Let $K$ be a global field of degree $d_{K}$ and $\mathcal{O}_{K}$ be its ring of integers. Set $Q=N^{\frac{\varepsilon}{2(d+1)}}$ and let $P\subseteq \mathcal{P}(Q)$ be a subset satisfying $w(P)\geq \kappa w(\mathcal{P}(Q))$. Then there exists a constant $C=C(d,k,\varepsilon,\alpha,\eta,K,D,M)$, such that for any set $S\subseteq [N]_{\mathbb{A}^{M+1}(\mathcal{O}_{K})}$ occupying less than $\alpha\mathcal{N}_{K}(\mathfrak{p})^{k}$ residue classes for every prime $\mathfrak{p}$, and that lies in an affine variety $Z\subseteq \mathbb{A}^{M+1}$ defined over $K$, geometrically irreducible of dimension $d+1$ and degree $D$, at least one of the following holds:
\begin{itemize}
\item ($S$ is small) $|S|\lesssim_{d,k,\varepsilon,K,D,M}N^{k-1+\varepsilon}$;
\item ($S$ is strongly algebraic) there exists a homogeneous polynomial  of degree at most $C$ 
and coefficients of height bounded by $N^C$ 
in  $\mathcal{O}_{K}[X_{0},\ldots ,X_{M}]$ vanishing at more than $(1-\eta)|S|$ points of $S$, that does not vanish at $Z$.
\end{itemize}
\label{reduction1}
\end{theorem}
\begin{remark}\label{observacion acerca de teorema 4.1}
There are two remarks to be made about the statement in Theorem \ref{reduction1}. 
First, in the case when $S$ is small, $d_{K}$ does not appear in the exponent of $N$. This is due to the fact that here we are using the relative height to $K$ instead of the absolute height as in Theorem \ref{basis3}. The reason for this is because it will simplify some of the cumbersome notation in our proofs, and also because it reflects more accurately the nature of the problem we are studying, which is relative to the global field $K$. The second remark is about the parameters $k,d$ and $M$; the choice of $M+1$ instead of $M$ and $Z$ of dimension $d+1$ instead of dimension $d$ are because we will lift a subset $S\subseteq Z\subseteq \mathbb{P}^{M}$ occupying less than $\lesssim \mathcal{N}_{K}(\mathfrak{p})^{k}$ residue classes for every prime $\mathfrak{p}$, with $Z$ a projective variety. Thus, the lifted set $\overline{S}$ will lie in the affine cone $C(Z)\subseteq \mathbb{A}^{M+1}$ which is a geometrically irreducible variety of dimension $d+1$ and degree $D$.
\end{remark}

\begin{proof}[Proof that Theorem \ref{reduction1} implies Theorem \ref{basis3}]
Let $S$ be as in the statement of Theorem \ref{basis3}. In particular, $|[S]_{\mathfrak{p}}|\leq \alpha \mathcal{N}_{K}(\mathfrak{p})^{k}$ for all prime $\mathfrak{p}$ and some positive $\alpha$ and $0\leq k<\dim (Z)$. By Proposition \ref{serre}, we can lift $S$ to a subset $\overline{S}\subseteq [cN]_{\mathbb{A}^{M+1}(\mathcal{O}_{K})}$ lying in the affine cone $C(Z)\subseteq \mathbb{A}^{M+1}$, where $c$ is some positive constant depending only on $K$. Note that $C(Z)$ has dimension $\dim(Z)+1$ and degree $\deg(Z)$. Now, let $\mathfrak{p}\in P$. Let us bound $[\overline{S}]_{\mathfrak{p}}$. Given ${\bf x}'=(x_{0},\ldots ,x_{M+1})\in \overline{S}$, there are two possibilities: $\mathfrak{p}|x_{i}$ for all $i$, or there exists $i_{0}$ such that $\mathfrak{p}\nmid x_{i_{0}}$. In the second case, $(x_{0}(\text{mod } \mathfrak{p}),\ldots ,x_{M+1}(\text{mod } \mathfrak{p}))$ defines a point in $\mathbb{P}^{M}(\mathcal{O}_{K}/\mathfrak{p})$ which coincides with the reduction modulo $\mathfrak{p}$ of ${\bf x}=(x_{0}:\ldots :x_{M+1})\in S$. Recall that each point in $\mathbb{P}^{M}(\mathcal{O}_{K}/\mathfrak{p})$ has $(\mathcal{N}_{K}(\mathfrak{p})-1)$ lifts to $\mathbb{A}^{M+1}(\mathcal{O}_{K}/\mathfrak{p})$.
Thus, the number of points in $[\overline{S}]_{\mathfrak p}$ that verify the second condition are bounded by 
$(\mathcal{N}_{K}(\mathfrak{p})-1)|[S]_{\mathfrak{p}}|\leq \alpha\mathcal{N}_{K}(\mathfrak{p})^{k+1}$. Meanwhile, in the first case, we just have that the points reduce to the $0$ class in $(\mathcal{O}_{K}/\mathfrak{p})^{M+1}$. We conclude that $|[\overline{S}]_{\mathfrak{p}}|\leq \alpha\mathcal{N}_{K}(\mathfrak{p})^{k+1}+1\leq \max\{\alpha,1\}\mathcal{N}_{K}(\mathfrak{p})^{k+1}=\alpha'\mathcal{N}_{K}(\mathfrak{p})^{k+1}$. Hence, $\overline{S}$ satisfies the hypothesis of Theorem \ref{reduction1} (with $k:=k+1$ and $\alpha:=\alpha'$) and at least one of the following holds:
\begin{itemize}
\item $|S|=|\overline{S}|\lesssim_{d,k,\varepsilon,K,D,M}(cN)^{(k+1)-1+\varepsilon}\lesssim_{d,k,\varepsilon,K,D,M}N^{k+\varepsilon}$;
\item there exists an homogeneous polynomial $f\in \mathcal{O}_{K}[X_{0},\ldots ,X_{M}]$ of degree at most $C=O_{d,k,\varepsilon,K,D,M,\eta}(1)$ 
and coefficients of height bounded by $N^C$ vanishing at more than $(1-\eta)|\overline{S}|=(1-\eta)|S|$ points of $\overline{S}$, that does not vanish at $C(Z)$.
\end{itemize}
From this we deduce Theorem \ref{basis3}.
\end{proof}

\subsection{An effective Noether's normalization} At this stage we have reduced Theorem \ref{basis3} to a problem about affine varieties. In order to carry on the second step discussed in the introduction of section $\S$\ref{section 3}, here we prove the following version of Noether's normalization theorem.

\begin{theorem}[Effective Noether's normalization theorem]
Let $V\subseteq \mathbb{P}^{m}$ be an irreducible projective variety defined over a global field $K$. Then there exists a finite map $\varphi:V\rightarrow \mathbb{P}^{\dim(V)}$, defined over $\mathbbm{k}$, such that
$$\varphi({\bf x})=(L_{0}({\bf x}):\ldots :L_{\dim(V)}({\bf x})),$$ 
with  $L_{0}({\bf x}),\ldots ,L_{\dim(V)}({\bf x})$ linear forms with coefficients on $\mathbbm{k}$ of height bounded by $\lesssim_{\mathbbm{k},m}(\deg(V))^{m-\dim(V)}$ where the implicit constant is effectively computable.
Moreover, each fibre of $\varphi$ has at most $\deg(V)$ elements. In particular, the same statement holds for an affine variety $Z\subseteq \mathbb{A}^{m}$.
\label{noether} 
\end{theorem}

\begin{proof}
Let $V$ be an irreducible variety as in the statement of the theorem. If $V=\mathbb{P}^{m}$, we take $\varphi$ the identity map and we are done. If $V\subsetneq \mathbb{P}^{m}$, then there exists ${\bf x}\in \mathbb{P}^{m}(\overline{\mathbbm{k}})\backslash V(\overline{\mathbbm{k}})$. Furthermore, let us see that we may choose ${\bf x}$ with coordinates in $K$, and of small height. If $\dim(V)=m-1$, $V$ is a hypersurface. If not, we reduce to the hypersurface case by means of a standard geometrical idea (see, for instance, Theorem 1 in \cite{MR0282975}, or its reprint in  \cite{Mumford0}). Indeed, for that, choose a generic projective subspace $W$ of $\mathbb{P}^{m}(\overline{\mathbbm{k}})$ of dimension $m-\dim(V)-2$, and consider the cone $C(W,V)$ formed by taking the union of all the lines joining a point in $W$ with a point in $V$. It is generically a projective subvariety of dimension $m-1$, i.e. an hypersurface. Furthermore, it has degree $\deg(V)$, so $C(W,V)$ is defined by a polynomial with coefficients on $\overline{\mathbbm{k}}$, of degree $\deg(V)$. In any case, we then have that $V$ is contained in an hypersurface $\mathcal{Z}(f)\subseteq \mathbb{P}^{m}(\overline{\mathbbm{k}})$, where $f\in \overline{\mathbbm{k}}[T_{0},\ldots ,T_{m}]$ is a non-zero homogeneous polynomial of degree $\deg(V)$.

Let us suppose that $f$ has a non-zero coefficient at $T_{0}^{d_{0}}\cdots T_{m}^{d_{m}}$. Consider the set $[\deg(V)]_{\mathcal{O}_{\mathbbm{k}}}=\{x\in \mathcal{O}_{\mathbbm{k}}:H(x)\leq \deg(V)\}$ (note that here we are using the absolute height instead of the relative height to $K$). It has strictly more than $\deg(V)\geq d_{i}$ elements. By the Combinatorial Nullstellensatz, Theorem 9.2 in \cite{Tao},  there exists $x_{0},\ldots, x_{n} \in [\deg(V)]_{\mathcal{O}_{\mathbbm{k}}}$ such that $f(x_{0},\ldots ,x_{m})\neq 0$. In particular $(x_{0},\ldots ,x_{m})\neq 0$. Let ${\bf x}_{1}\in \mathbb{P}^{m}$ be the point with projective coordinates $(x_{0}:\ldots :x_{m})$. By construction, ${\bf x}_{1}\in \mathbb{P}^{m}(\mathbbm{k})\backslash V(\mathbbm{k})$. Now, construct linear forms $L_{1,1}(T_{0},\ldots ,T_{m}),\ldots, L_{1,m}(T_{0},\ldots ,T_{m})\in \mathbbm{k}[T_{0},\ldots ,T_{m}]$ such that $\mathcal{Z}(L_{1,1},\ldots ,L_{1,m})$ is equal to $\{{\bf x}_{1}\}$. Note that the coefficients of such linear forms are a basis of the vector space $V_{1}:=\left\langle (x_{0},\ldots ,x_{m})\right\rangle^{\perp}$ defined over $\mathbbm{k}$. Thus, if we want to construct the linear forms $L_{1,i}$'s with coefficients of small height, it is enough to find a basis of small height of $V_{1}$. For this, we use the generalizations of Siegel's lemma for number fields and function fields in \cite{Bombieri2, Thunder, Fu} to find a basis ${\bf y}_{1},\ldots ,{\bf y}_{m}\in K^{m+1}$ of $V_{1}$ such that
\begin{equation}
\displaystyle \prod_{i=1}^{m}H(1:{\bf y}_{i}) \lesssim_{\mathbbm{k},m}H(V_{1}),
\end{equation}
where $H(V_{1})$ is the height of $V_{1}$. By  duality (see, for instance, Duality Theorem in \cite{Thunder0}), it holds that $H(V_{1})$ coincides with the height of the linear subspace generated by the $\mathbbm{k}$-vector $(x_{0},\ldots ,x_{m})$. Moreover this height coincides with the projective height $H(x_{0}:\ldots :x_{m})$. Thus
\begin{equation}
\displaystyle \prod_{i=1}^{m}H(1:{\bf y}_{i}) \lesssim_{\mathbbm{k},m} H(x_{0}:\ldots :x_{m})  \lesssim_{\mathbbm{k},m} \deg(V).
\end{equation}
We define $\varphi_{1}:V\rightarrow \mathbb{P}^{m-1}$ as the projection away from ${\bf x}_{1}$, that is 
\begin{equation}
\varphi_{1}({\bf x})=(L_{1,1}({\bf x}):\ldots :L_{1,m}({\bf x})).
\end{equation}
Thus $\varphi_{1}$ is a finite morphism (see Theorem 7 from Section 5.3 in Chapter 1 of \cite{Shafarevich}), with $L_{1,1},\ldots ,L_{1,m}$ linear forms with coefficients of height bounded by $\lesssim_{\mathbbm{k},m}\deg(V)$. If $\varphi_{1}(V)=\mathbb{P}^{m-1}$, we are done. Now suppose otherwise. Then for a generic linear space $L\subseteq \mathbb{P}^{m-1}$ of codimension $\dim(V)$, it holds that $L\cap \varphi_{1}(V)$ is finite. Furthermore, the pre-image of $L\cap \varphi_{1}(V)$ by $\varphi_{1}$ is finite (because $\varphi_{1}$ is a finite morphism, hence it has finite fibers), and this intersection is equal to the intersection of $V$ by some linear subspace $L'\subseteq \mathbb{P}^{m}$ of codimension $\dim(V)$ (a finite morphism preserves the dimension). Thus $|L\cap \varphi_{1}(V)|\leq |L'\cap V|\leq \deg(V)$, from where we conclude that the degree of $\varphi_{1}(V)$ is at most $\deg(V)$.   

In conclusion, the projective irreducible variety $\varphi_{1}(V)$ has dimension $\dim(V)$ and $\deg(\varphi_{1}(V))\leq \deg(V)$. Hence we can repeat the same argument as above and obtain a sequence of finite maps $\varphi_{i+1}:\varphi_{i}(V)\rightarrow \mathbb{P}^{m-i+1}$, defined by
\begin{equation}
\varphi_{i+1}({\bf x})=(L_{i+1,1}({\bf x}):\ldots :L_{i+1,m-i+1}({\bf x})),
\end{equation}
with $L_{i+1,1},\ldots ,L_{i+1,m-i+1}$ linear forms with coefficients in $\mathbbm{k}$ of height bounded by $\lesssim_{\mathbbm{k},m}\deg(V)$. Since the sequence $\varphi_{1},\varphi_{2},\ldots $ ends with $i=m-\dim(V)$, we conclude that there exists a finite morphism $\varphi:V\rightarrow \mathbb{P}^{\dim(V)}$ such that
\begin{equation}
\varphi({\bf x})=(L_{0}({\bf x}):\ldots :L_{\dim(V)}({\bf x})),
\end{equation}
with $L_{0},\ldots ,L_{m}$ linear forms with coefficients in $\mathbbm{k}$ of height bounded by $\lesssim_{\mathbbm{k},m}\deg(V)^{m-\dim(V)}$.
 
Finally, note that the morphism $\varphi:V\rightarrow \mathbb{P}^{\dim(V)}$ that we constructed can be interpreted geometrically as the projection away from a generic linear subspace $L\subseteq \mathbb{P}^{m}$ of codimension $\dim(V)+1$. Thus, given ${\bf z}\in \mathbb{P}^{\dim(V)}$ the points in the fibre $\varphi^{-1}({\bf z})$ correspond to the points lying in the intersection of the affine cone $C(V)\subseteq \mathbb{A}^{m+1}$ of $V$ with respect to an affine linear subspace $L'\subseteq \mathbb{A}^{m+1}$. Since the affine cone $C(V)$ has the same degree of $V$, we conclude that $|\varphi^{-1}({\bf z})|\leq |C(Z)\cap L'|\leq \deg(V)$.
\end{proof}

\begin{remark}
Professor Sombra informed us that in the literature there are other effective versions of the Noether's normalization theorem for number fields. For instance, Lemma 2.14 and Proposition 4.5 in  \cite{Sombra}  imply that for an affine variety $V$ there exist linear forms in Noether position with height bounded by $2\deg(V)^{2}$. We wish to thank Prof. Sombra and Prof. Dickenstein for this.
\label{remark about noether}
\end{remark}

\subsection{Reduction to the affine plane case} Now we are in condition to make the last reduction of Theorem \ref{basis3}. Specifically, by means of an effective change of variables, we will reduce 
the proof of Theorem \ref{reduction1} to proving the next statement.
\begin{theorem}[Case $Z=\mathbb{A}^{d+1}$]
Let $0\leq k<d$ be integers, and let $\varepsilon, \alpha,\kappa,\eta>0$ be positive real numbers. Let $K$ be a global field of degree $d_{K}$ and $\mathcal{O}_{K}$ be its ring of integers. Set $Q=N^{\frac{\varepsilon}{2(d+1)}}$ and let $P\subseteq \mathcal{P}(Q)$
be a subset satisfying $w(P)\geq \kappa w(\mathcal{P}(Q))$. Then there exists a constant $C=C(d,k,\varepsilon,\alpha,\eta,K)$, such that for any set $S\subseteq [N]_{\mathcal{O}_{K}}^{d+1}$ occupying less than $\alpha\mathcal{N}_{K}(\mathfrak{p})^{k}$ residue classes for every prime $\mathfrak{p}$, at least one of the following holds:
\begin{itemize}
\item ($S$ is small) $|S|\lesssim_{d,k,\varepsilon,K}N^{k-1+\varepsilon}$;
\item ($S$ is strongly algebraic) there exists a non-zero homogeneous polynomial of degree at most $C$
and coefficients of height bounded by $N^C$
 in $\mathcal{O}_{K}[X_{0},\ldots ,X_{d}]$ vanishing at more than $(1-\eta)|S|$ points of $S$.
\end{itemize}
\label{reduction2}
\end{theorem}
\begin{proof}[Proof that Theorem \ref{reduction2} implies Theorem \ref{reduction1}]
Let $S\subseteq Z\subseteq \mathbb{A}^{M+1}$ be as in Theorem \ref{reduction1}. By Theorem \ref{noether} there exists a finite map ${\bf F}=(F_{1},\ldots ,F_{d+1}):Z\rightarrow \mathbb{A}^{d+1}$ such that each for all $i$, $F_{i}\in \mathcal{O}_{\mathbbm{k}}[X_{0},\ldots ,X_{M}]$ is a linear form with coefficients of height bounded by $\lesssim_{\mathbbm{k},M,d,D}1$. Note that by \eqref{polynomialheight2} and \eqref{obvio} it holds
 ${{\bf F}(S)} \subseteq [cN]_{\mathbb{A}^{d+1}(\mathcal{O}_{K})}\subseteq [cN]_{\mathcal{O}_{K}}^{d+1}$ for some $c=O_{K,M,d,D}(1)$. Since ${\bf F}$ preserves congruences, for any prime $\mathfrak{p}\in P$ it holds $|[{\bf F}(S)]_{\mathfrak{p}}|\leq |[S]_{\mathfrak{p}}|\leq \alpha\mathcal{N}_{K}(\mathfrak{p})^{k}$ with $k<d+1$. Hence, ${\bf F}(S)$ is in the conditions of Theorem \ref{reduction2}, so apply the theorem to ${\bf F}(S)$ with $\eta=\frac{1}{2}$ to conclude that
\begin{itemize}
\item $|{\bf F}(S)|\lesssim_{d,k,\varepsilon,K}N^{k-1+\varepsilon}$; or
\item there exists an homogeneous polynomial $g\in \mathcal{O}_{K}[Y_{0},\ldots ,Y_{d}]$ of degree at most $C$ with coefficients of height bounded by $N^{C}$,
 vanishing at more than $\frac{1}{2}|{\bf F}(S)|$ points of ${\bf F}(S)$.
\end{itemize}
Suppose that the first possibility occurs. Since ${\bf F}$ has degree at most $\deg(Z)=D$, we have $|{\bf F}(S)|\geq \frac{|S|}{\deg(Z)}$, so we deduce $|S|\lesssim_{d,k,\varepsilon,K,D}N^{k-1+\varepsilon}$. If the second possibility occurs, using again that ${\bf F}$ has degree at most $\deg(Z)=D$, we conclude that $g\in \mathcal{O}_{K}[Y_{0},\ldots ,Y_{d}]$ is an homogeneous polynomial of degree at most $C$ vanishing at more than $\frac{1}{2}|{\bf F}(S)|\geq \frac{1}{2D}|S|:=\eta_{0}|S|$ points of ${\bf F}(S)$. Let $f(X_{0},\ldots ,X_{M}):=g({\bf F}(X_{0},\ldots ,X_{M}))$. Since $g\neq 0$ and ${\bf F}$ is surjective, we conclude that $f$ is an homogeneous polynomial of degree at most $C$, vanishing on at least $\eta_{0}|S|$ points of $S$, that does not vanish at $Z$. Furthermore, since $f$ is the composition of the polynomial functions ${\bf F}$ and $g$, and their coefficients are bounded by $\lesssim_{\mathbbm{k},M,d,D}1$ and $N^{C}$ respectively, we see that there exists a constant $C'=C'(d,k,\varepsilon,\alpha,\eta_{0},K,D,M)$, such that $f$ has degree bounded by $C'$, and coefficients of height bounded by $N^{C'}$. Thus, we conclude Theorem \ref{reduction1} for $\eta_{0}=\frac{1}{2D}$. We conclude Theorem \ref{reduction1} for any $\eta$ by a simple partition argument.\end{proof}

\section{The inverse sieve problem in the affine space over a global field}
\label{section 4}

In this section we are going to deduce Theorem \ref{reduction2} as a consequence of the following stronger version of it.

\begin{theorem}
Let $d\in\mathbb N$ and $h>1$ and let $\varepsilon,\eta>0$ be positive real numbers. Let $K$ be a global field of degree $d_{K}$ and $\mathcal{O}_{K}$ be its ring of integers. Set $Q=N^{\frac{\varepsilon}{2d}}$ and let $P\subseteq \mathcal{P}(Q)$ satisfying $w(P)\geq \kappa w(\mathcal{P}(Q))$ for some $\kappa>0$. Suppose that $S\subseteq [N]_{\mathcal{O}_{K}}^{d}$ is a set of size $|S|\geq cN^{d-h-1+\varepsilon}$ occupying at most $\alpha \mathcal{N}_{K}(\mathfrak{p})^{d-h}$ residue classes modulo $\mathfrak{p}$ for every prime $\mathfrak{p}\in P$ and some $\alpha>0$. Then if $N$ is sufficiently large there exists a non-zero homogeneous polynomial $f\in \mathcal{O}_{K}[X_{1},\ldots ,X_{d}]$ of degree $O_{d, h,\varepsilon,\eta, \kappa, K}(1)$ and coefficients of height bounded by $N^{O_{d,h,\varepsilon,\eta,\kappa,K}(1)}$ which vanishes at more that $(1-\eta)|S|$ points of $S$.
\label{theorem 2.4}
\end{theorem}

Observe that Theorem \ref{reduction2} has been stated for $S\subseteq [N]_{\mathcal{O}_{K}}^{d+1}$ rather than for
$S\subseteq [N]_{\mathcal{O}_{K}}^{d}$. The reason to do this has been explained in Remark \ref{observacion acerca de teorema 4.1}. 
With this in mind,  the change of variables $k:=d-h$ (observe that we must take $h>1$ so that $k<d-1$ is as in the statement of  Theorem \ref{reduction2}), immediately shows that Theorem \ref{theorem 2.4} implies Theorem \ref{reduction2}.

In order to prove Theorem \ref{theorem 2.4}, we follow the proof of Theorem 2.4 in \cite{Walsh2}. The idea of the proof is to construct a small ``characteristic subset'' $A\subseteq S$ with the property that any ``small'' polynomial that vanishes at $A$ also vanishes at a positive proportion of $S$. This will be done in Proposition \ref{proposition 2.2}, which is adapted from Proposition 2.2 in \cite{Walsh2}.
In order to achieve this, we want to do induction on $d$, with $h$ fixed. For that, we will first show that a large proportion of the fibers of $S$ are large and badly distributed. By the inductive hypothesis, each of these fibers contains a small characteristic subset. The next technical step will be to glue few of these characteristic subsets to obtain the small characteristic subset $A$ of $S$.
 Then, the small size of $A$ will enable us to use a variant of Siegel's lemma to construct a small homogeneous  polynomial that vanishes on $A$. 
 
 In most of the proof of Theorem \ref{theorem 2.4}, $h$ can be taken to be any positive integer. It is only in section $\S$\ref{4.3}, where we construct such a small homogeneous polynomial, when we need to use that $h>1$.
 Albeit the proof presented here follows the same steps of the proof of Walsh, some new technical difficulties arise, that were already discussed in the introduction of this article. In order to aid the reader, we closely follow the notations and definitions used in \cite{Walsh2}.
  
 \subsection{Genericity}

First we extend the notion of ``generic set'' of \cite{Walsh2}, and then prove that any set satisfying the conditions of Theorem \ref{theorem 2.4} contains large generic subsets. For this, we introduce the following notation. Given $S\subseteq [N]_{\mathcal{O}_{K}}^{d}$, ${\bf a}\in (\mathcal{O}_{K}/\mathfrak{p})^{d}$, $a\in \mathcal{O}_{K}/\mathfrak{p}$ and $x \in [N]_{\mathcal{O}_{K}}$, 
\begin{align*}
S({\bf a},\mathfrak{p})&:=\{{\bf x}=(x_{1},\ldots ,x_{d})\in S:\pi_{\mathfrak{p}}({\bf x})={\bf a}\},\\
S(a,\mathfrak{p})&:=\{{\bf x}=(x_{1},\ldots ,x_{d})\in S:x_{1}\equiv a(\text{mod}\ \mathfrak{p})\},\\
S_{x}&:=S\cap \pi_{1}^{-1}(x).
\end{align*}
\begin{definition}[Genericity]
Given a real number $B>0$ and some integer $l\geq0$ we say that a set $S\subseteq [N]_{\mathcal{O}_{K}}^{d}$ is $(B,l)$-generic modulo $\mathfrak{p}$ if
$$\dfrac{|S({\bf a},\mathfrak{p})|}{|S|}<\dfrac{B}{\mathcal{N}_{K}(\mathfrak{p})^{l}},$$
for every residue class ${\bf a}\  \text{modulo}\ \mathfrak{p}$.
\end{definition}
The notion of genericity will capture the fact that the sections of $S$, on average,  share many residue classes modulo $\mathfrak{p}$
 for many primes $\mathfrak{p}$.

\begin{lemma}[{Compare to Lemma 3.4 in \cite{Walsh2}}]
Let $d,h\geq 1$ be arbitrary integers and let $\varepsilon>0$ be some positive real number. Let $K$ be a global field of degree $d_{K}$ and $\mathcal{O}_{K}$ be its ring of integers. Set $Q=N^{\frac{\varepsilon}{2d}}$ and let $P\subseteq \mathcal{P}(Q)$ satisfying $w(P)\geq \kappa w(\mathcal P(Q))$ for some $\kappa>0$. Suppose that $S\subseteq [N]_{\mathcal{O}_{K}}^{d}$ is a subset of size at least $cN^{d-h-1+\varepsilon}$ occupying at most $\alpha\mathcal{N}(\mathfrak{p})^{d-h}$ residue classes modulo $\mathfrak{p}$ for all prime $\mathfrak{p}\in P$ and some $\alpha>0$.
Then if $N$ is sufficiently large, there exist constants $B=B(d,h,\varepsilon,\kappa,\alpha,c,K)$, $\kappa_{1}=\kappa_{1}(d,h,\varepsilon,\kappa,\alpha,c,K)$, $c_{1}=c_{1}(d,h,\varepsilon,\kappa,\alpha,c,K)$ such that there exists a subset of primes $P'\subseteq P$ with $w(P')\geq \kappa_{1}w(P)$ such that for each $\mathfrak{p}\in P'$ there is some subset $\mathcal{G}_{\mathfrak{p}}(S)\subseteq S$ with $|\mathcal{G}_{\mathfrak{p}}(S)|\geq c_{1}|S|$, which is $(B,d-h)$-generic modulo $\mathfrak{p}$.
\label{lemma 3.4}
\end{lemma}

\begin{proof}
From now on let us fix an integer $h\geq 1$.  If $d\leq h$, then we may take $B=2$, $\mathcal{G}_{\mathfrak{p}}(S)=S$ and $P'=P$. Thus
\begin{equation}
B=B(d,h,\varepsilon,\kappa,\alpha,c,K)=2,
\label{B0}
\end{equation}
\begin{equation}
\kappa_{1}=\kappa_{1}(d,h,\varepsilon,\kappa,\alpha,c,K)=1,
\label{kappa10}
\end{equation}
\begin{equation}
c_{1}=c_{1}(d,h,\varepsilon,\kappa,\alpha,c,K)=1.
\label{c10}
\end{equation}
We now proceed by induction on $d$. Let us suppose that $d\geq h+1$ is an integer and let us assume that Lemma \ref{lemma 3.4} holds for every smaller dimension. Let $S$ and $P$ be as in the statement. For each $1\leq i\leq d$, define $\pi_{i}:K^{d}\rightarrow K$ as the projection in the $i$-th coordinate. 
\begin{claim}
There exists a constant $C_{3}=C_{3}(d,h,\varepsilon,\kappa,\alpha,c)$ such that if $N\geq C_{3}$ then there exist $1\leq i\leq d$ and a subset $\tilde{S}\subseteq S$ with $|\tilde{S}|\geq \frac{|S|}{2^{d}}$ such that for any $A\subseteq \tilde{S}$ with $|A|\geq \frac{|\tilde{S}|}{2}$ we have $|\pi_{i}(A)|\geq Q$.
\label{claim1}
\end{claim}  
\begin{proof}[Proof of Claim \ref{claim1}]
Let us suppose that the claim is false for $\tilde{S}=S$ and $i=1$. Hence there exists $A_{1}\subseteq S$ with $|A_{1}|\geq \frac{|S|}{2}$ and $|\pi_{1}(A_{1})|<Q$. Then, if the claim fails for $\tilde{S}=A_{1}$ and $i=2$, there exists $A_{2}\subseteq A_{1}$ with $|A_{2}|\geq \frac{|A_{1}|}{2}\geq \frac{|S|}{2^{2}}$ and $|\pi_{1}(A_{2})|,|\pi_{2}(A_{2})|<Q$. Iterating this process $d$ times, either we get the claim or end up with a subset $A_{d}\subseteq S$ with $|A_{d}|\geq \frac{|S|}{2^{d}}$. Thus,
\begin{equation}
cN^{d-h-1+\varepsilon}\leq |S|\leq 2^{d}|A_{d}|\leq 2^{d}\displaystyle \prod_{i=1}^{d}|\pi_{i}(A_{d})|<2^{d}Q^{d}=2^{d}N^{\frac{\varepsilon}{2}}.
\label{absurd}
\end{equation}
Now, \eqref{absurd} is absurd if $N\geq \left(\dfrac{2^{d}}{c}\right)^{\frac{1}{d-h-1+\frac{\varepsilon}{2}}}$.
\end{proof}
By  Claim \ref{claim1}, at the cost of passing to a subset of $S$ of density at least equal to $\frac{1}{2^d}$  and permuting the variables, from now on we assume that S satisfies  
 \begin{equation}
 |\pi_{1}(A)|\geq Q\text{ for all }A\subseteq S\text{ with }|A|\geq \dfrac{|S|}{2}.
 \label{3.4}
 \end{equation}
 
 The purpose now is to find a dense subset of $S$ which is in condition to apply the inductive hypothesis over the fibers. Afterwards, the generic sets of the fibers of this dense subset will be glued together to obtain the desired generic set. In order to do so we will first eliminate some problematic fibers. Let $\mathfrak{p}$ a prime in $P$. Recall that for any $a\in \mathcal{O}_{K}/\mathfrak{p}$, $S(a,\mathfrak{p})$ denotes the elements ${\bf x}$ of $S$ for which $\pi_{1}({\bf x})\equiv a(\text{mod}\ \mathfrak{p})$. Let $B_{1}$ be a constant sufficiently large to be chosen later. Since $|[S]_{\mathfrak{p}}|\leq \alpha\mathcal{N}_{K}(\mathfrak{p})^{d-h}$, it is clear that there are at most $\frac{\alpha}{B_{1}}\mathcal{N}_{K}(\mathfrak{p})$ residue classes $a\in [\pi_{1}(S)]_{\mathfrak{p}}\subseteq \mathcal{O}_{K}/\mathfrak{p}$ for which $|[S(a,\mathfrak{p})]_{\mathfrak{p}}|\geq B_{1}\mathcal{N}_{K}(\mathfrak{p})^{d-h-1}$. Let us denote
 \begin{equation}
 \mathcal{E}_{1}(\mathfrak{p}):=\left\{ a\in [\pi_{1}(S)]_{\mathfrak{p}}:|[S(a,\mathfrak{p})]_{\mathfrak{p}}|\geq B_{1}\mathcal{N}_{K}(\mathfrak{p})^{d-h-1} \right\}.
 \end{equation}
 We will also write
 \begin{equation}\label{definicion de e2}
 \mathcal{E}_{2}(\mathfrak{p}):=\left\{ a\in [\pi_{1}(S)]_{\mathfrak{p}}:|S(a,\mathfrak{p})|\geq \dfrac{B_{1}}{\alpha\mathcal{N}_{K}(\mathfrak{p})}|S| \right\}.
 \end{equation}
 From the identity $\sum_{a\in \mathcal{O}_{K}/\mathfrak{p}}|S(a,\mathfrak{p})|=|S|$ it follows that $|\mathcal{E}_{2}(\mathfrak{p})|\leq \frac{\alpha}{B_{1}}\mathcal{N}_{K}(\mathfrak{p})$, hence if $\mathcal{E}(\mathfrak{p}):=\mathcal{E}_{1}(\mathfrak{p})\cup \mathcal{E}_{2}(\mathfrak{p})$  is the set of these exceptional residue classes, we have $|\mathcal{E}(\mathfrak{p})|\leq \frac{2\alpha}{B_{1}}\mathcal{N}_{K}(\mathfrak{p})$.  We will use the larger sieve (Lemma \ref{lemma 3.1}) to prove that few $x\in [N]_{\mathcal{O}_{K}}$ verify that $x (\text{mod}\ \mathfrak{p})$ lie in $\mathcal{E}(\mathfrak{p})$ for many $\mathfrak{p}\in P$. Indeed, let us consider the set 
\begin{equation}\label{definicion de X mayuscula}
X:=\left\lbrace x\in [N]_{\mathcal{O}_{K}}: \sum_{\mathfrak{p}\in P}1_{x (\text{mod}\ \mathfrak{p})\in \mathcal{E}(\mathfrak{p})}\dfrac{\log(\mathcal{N}_{K}(\mathfrak{p}))}{\mathcal{N}_{K}(\mathfrak{p})}\geq \dfrac{1}{2}w(P) \right\rbrace.
\end{equation}
Let $P_{1}\subseteq P$ the set of primes such that at least $\frac{1}{4}|X|$ elements of $X$ lie in the exceptional set of residue classes $\mathcal{E}(\mathfrak{p})$, namely, $P_{1}:=\left\lbrace \mathfrak{p} \in P:  \left| \bigcup_{a\in \mathcal{E}(\mathfrak{p})}X(a,\mathfrak{p})\right| \geq \frac{1}{4}|X| \right\rbrace.$
\begin{claim}
$w(P_{1})\geq \frac{1}{4}w(P).$
\label{claim2}
\end{claim}
\begin{proof}[Proof of Claim \ref{claim2}]
Observe that

$$\displaystyle \sum_{a\in \mathcal{E}(\mathfrak{p})}|X(a,\mathfrak{p})|=\left| \displaystyle \bigcup_{a\in \mathcal{E}(\mathfrak{p})}X(a,\mathfrak{p}) \right|=\sum_{x\in X}1_{x (\text{mod}\ \mathfrak{p})\in \mathcal{E}(\mathfrak{p})}.$$
On the one hand we have 
\begin{align}\label{1}
\displaystyle \sum_{\mathfrak{p}\in P}\sum_{x\in X}1_{x(\text{mod}\ \mathfrak{p})\in \mathcal{E}(\mathfrak{p})}\dfrac{\log(\mathcal{N}(\mathfrak{p}))}{\mathcal{N}(\mathfrak{p})}&=\sum_{x\in X}\left( \sum_{\mathfrak{p}\in P}1_{x(\text{mod}\ \mathfrak{p})\in \mathcal{E}(\mathfrak{p})}\dfrac{\log(\mathcal{N}(\mathfrak{p}))}{\mathcal{N}(\mathfrak{p})} \right)\\
&\geq \sum_{x\in X}\dfrac{1}{2}w(P) \nonumber\\
&=\dfrac{1}{2}w(P)|X|.\nonumber
\end{align}
On the other hand,
\begin{align}\label{2}
\displaystyle \sum_{\mathfrak{p}\in P}&\sum_{x\in X}1_{x (\text{mod}\ \mathfrak{p})\in \mathcal{E}(\mathfrak{p})}\dfrac{\log(\mathcal{N}(\mathfrak{p}))}{\mathcal{N}(\mathfrak{p})} \\
& =\sum_{\mathfrak{p}\notin P_{1}}\left(\sum_{a\in \mathcal{E}(\mathfrak{p})}|X(a,\mathfrak{p})|\right)\dfrac{\log(\mathcal{N}(\mathfrak{p}))}{\mathcal{N}(\mathfrak{p})}+\sum_{\mathfrak{p}\in P_{1}}\left(\sum_{a\in \mathcal{E}(\mathfrak{p})}|X(a,\mathfrak{p})|\right)\dfrac{\log(\mathcal{N}(\mathfrak{p}))}{\mathcal{N}(\mathfrak{p})}\nonumber\\
& < \sum_{\mathfrak{p}\notin P_{1}}\dfrac{1}{4}|X|\dfrac{\log(\mathcal{N}(\mathfrak{p}))}{\mathcal{N}(\mathfrak{p})} +\sum_{\mathfrak{p}\in P_{1}}|X|\dfrac{\log(\mathcal{N}(\mathfrak{p}))}{\mathcal{N}(\mathfrak{p})}\nonumber \\
& =\dfrac{1}{4}|X|w(P\backslash P_{1})+|X|w(P_{1}) \nonumber \\ 
& =\dfrac{1}{4}|X|w(P)- \dfrac{1}{4}|X|w(P_{1})+|X|w(P_{1}).\nonumber
\end{align}
Comparing \eqref{1} and \eqref{2}, we conclude that $w(P_{1}) \geq \frac{1}{4}w(P)$.
\end{proof}
In Lemma \ref{lemma 3.1}, setting the constants $\gamma:=\frac{\varepsilon}{2d}$, $\kappa:=\frac{\kappa}{4}$, $\mu:=\frac{1}{4}$, $\alpha:=\frac{2\alpha}{B_{1}}$, it follows that if we chose 
\begin{equation}
B_{1}:=\dfrac{2\alpha}{C_{1}\left( \dfrac{\kappa}{4}, \dfrac{1}{4}, \dfrac{\varepsilon}{2d},K \right)}=\frac{2^{8}\alpha c_{5,K}d}{\kappa\varepsilon}
\label{C1}
\end{equation}
then $|X|<Q$.
From \eqref{3.4} we deduce that $|S\backslash \pi_{1}^{-1}(X)|\geq \frac{1}{2}|S|$.\\

The next claim shows that there exists a dense subset of $S$ whose sections are in the conditions to apply the inductive hypothesis.
\begin{claim}
Let $0<\nu<1$. There exists a subset $S'\subseteq S$ with $|S'|\geq \frac{1}{4}|S|$ which does not intersect $\pi_{1}^{-1}(X)$ and such that for all $x\in \pi_{1}(S')$, $S'_{x}:=\pi_{1}^{-1}(x)\cap S'$ satisfies $|S'_{x}|\geq c'(K,\nu)N^{d-h-2+\nu\varepsilon}$ , where $c'(K,\nu)$ is a positive constant dependent on $K,\nu$.
\label{claim3}
\end{claim}
\begin{proof}[Proof of Claim \ref{claim3}]
The proof will be by contradiction. Let us assume that for all $S'\subseteq S\backslash \pi_{1}^{-1}(X)$ with $|S'|\geq \frac{1}{4}|S|$ satisfies $|S'_{x}|<c'(K,\nu)N^{d-h-2+\nu\varepsilon}$ for some $x\in \pi_{1}(S')$. Let
$$\overline{S}:=\left\{ s\in S\backslash \pi_{1}^{-1}(X): |(S\backslash \pi_{1}^{-1}(X))_{\pi_{1}(s)}|<c'(K,\nu)N^{d-h-2+\nu\varepsilon}\right\}.$$
Note that $\overline{S}$ and $S\backslash \pi_{1}^{-1}(X)$ have the same sections (whenever the section of $\overline{S}$ is nonempty). Thus, from the assumption we deduce that $|\overline{S}|\geq \frac{1}{2}|S\backslash \pi_{1}^{-1}(X)|$ and moreover $|\overline{S}_{x}|\leq c'(K,\nu)N^{d-h-2+\nu\varepsilon}$ for all $x\in \pi_{1}(\overline{S})$. Hence, 
\begin{align}\label{fact 4.5}
\dfrac{1}{4}cN^{d-h-1+\varepsilon}&\leq \dfrac{1}{4}|S|\leq |\overline{S}|=\left| \displaystyle \bigcup_{x\in \pi_{1}(\overline{S})}\overline{S}_{x} \right|\\
&<c'(K,\nu)N^{d-h-2+\nu\varepsilon}|\pi_{1}(\overline{S})|\nonumber\\
&\leq c'(K,\nu)N^{d-h-2+\nu\varepsilon}|[N]_{\mathcal{O}_{K}}|.\nonumber
\end{align}
Using Proposition \ref{S-integers bound} in \eqref{fact 4.5}, and using the bound $\log(N)\leq \frac{|M_{K,\infty}|}{(1-\nu)\varepsilon}N^{\frac{(1-\nu)\varepsilon}{|M_{K,\infty}|}}$, we arrive at the inequality
\begin{align*}
\dfrac{1}{4}cN^{d-h-1+\varepsilon}&\leq  c'(K,\nu)N^{d-h-1+\nu\varepsilon}c''(K)(\log(N))^{|M_{K,\infty}|}\\
&\leq c'(K,\nu)c''(K)\left(\dfrac{d_{K}}{(1-\nu)\varepsilon}\right)^{d_{K}}N^{d-h-1+\varepsilon}.
\end{align*}
But this inequality does not hold for 
\begin{equation}\label{definicion de c prima de K nu}
c'(K,\nu):=\dfrac{c}{8c''(K)}\left( \dfrac{(1-\nu)\varepsilon}{d_{K}} \right)^{d_{K}}.
\end{equation}
\end{proof}

Take $S'$ as in Claim \ref{claim3}. Every $x\in \pi_{1}(S')$ lies outside of $X$, thus by the definition of $X$ given in \eqref{definicion de X mayuscula}, each
$x\in \pi_{1}(S')$  has associated the subset of primes 
\begin{equation}
P_{x}:=\{\mathfrak{p}\in P: x(\text{mod}\ \mathfrak{p})\notin \mathcal{E}(\mathfrak{p})\}
\end{equation}
that verifies  $w(P_{x})\geq \frac{1}{2}w(P)\geq \frac{1}{2}\kappa w(\mathcal P(Q))$. 
Since $\mathcal{E}_{1}(\mathfrak{p})\subseteq \mathcal{E}(\mathfrak{p})$, we have that for each $x\in \pi_{1}(S')$ and for each $\mathfrak{p}\in P_{x}$, $|[S'_{x}]_{\mathfrak{p}}|\leq B_{1}\mathcal{N}_{K}(\mathfrak{p})^{d-h-1}$. Also, by the definition of $S'$ in Claim \ref{claim3}, $|S'_{x}|\geq c'(K,\nu)N^{d-h-2+\nu\varepsilon}$. 
 Since we are doing induction on $d$, at the step $d-1$, the parameter $Q$ should be changed to $N^{\frac{\nu\varepsilon}{2(d-1)}}$. Hence if we take 
  $\nu:=\frac{d-1}{d}<1$, the parameter $Q$ remains unchanged throughout the proof.

We are in condition to apply the inductive hypothesis to each $S'_{x}$ with $x\in \pi_{1}(S')$, to conclude that there exists a subset $P'_{x}\subseteq P_{x}$ with $w(P'_{x})\geq \kappa_{1}w(P_{x})$,
and constants $c_{1},B$ independent of $x$, such that for each $\mathfrak{p}\in P'_{x}$ there is a $(B,d-h-1)$-generic modulo $\mathfrak{p}$ set $\mathcal{G}_{\mathfrak{p}}(S'_{x})\subseteq S'_{x}$, containing at least $c_{1}|S'_{x}|$ elements. Here, the dependence of the constants are:
\begin{align}
B&=B\left(d-1,h,\nu\varepsilon,\frac{\kappa}{2},B_{1},c'(K,\nu),K\right),
\\
\kappa_{1}&=\kappa_{1}\left(d-1,h,\nu\varepsilon,\frac{\kappa}{2},B_{1},c'(K,\nu),K\right),\label{como es kappa1}
\\
c_{1}&=c_{1}\left(d-1,h,\nu\varepsilon,\frac{\kappa}{2},B_{1},c'(K,\nu),K\right),
\\
\nu&=\frac{d-1}{d},\label{nueva nu}
\end{align}
where $B_{1}$ and $c'(K,\nu)$ were determined in equations \eqref{C1} and \eqref{definicion de c prima de K nu} respectively.

Each fiber $S'_{x}$ has its own set of primes $P'_{x}$, with density $w(P'_{x})\geq \kappa_{1}w(P_{x})$. The fact that $\kappa_{1}$ is independent of $x$, will allow us to find a set of primes $P^{\prime}\subseteq P$ with $w(P^{\prime})\gtrsim w(P)$ such that for each $\mathfrak{p} \in P^{\prime}$ it will exist a generic set modulo $\mathfrak{p}$, $\mathcal{G}_{\mathfrak{p}}(S)$, constructed by gluing the generic sets modulo $\mathfrak{ p}$ of the fibers $S'_{x}$. We start by constructing the set $P^{\prime}$ in the next claim.

\begin{claim}
There exist a positive constant $\beta>0$ depending on the $\kappa_1$ given by \eqref{como es kappa1}, and a set of primes $P^{\prime}\subseteq P $ with $w(P^{\prime})\geq \frac{\kappa_{1}}{4}w(P)$ such that for each $\mathfrak{p}\in P^{\prime}$ there are at least $\beta|S'|$ elements $s\in S'$ for which $\mathfrak{p}\in P'_{\pi_{1}(s)}$.
\label{claim 4.6}
\end{claim}
\begin{proof}[Proof of Claim \ref{claim 4.6}]
Let $\beta>0$ and consider the set 
$$P^{\prime}:=\left\lbrace \mathfrak{p} \in P:  \sum_{s\in S'}1_{\mathfrak{p}\in P'_{\pi_{1}(s)}} \geq \beta|S'| \right\rbrace. $$
Then,
\begin{equation*}
\begin{split}
\sum_{s\in S'}\sum_{\mathfrak{p}\in P}1_{\mathfrak{p}\in P'_{\pi_{1}(s)}}\dfrac{\log(\mathcal{N}_{K}(\mathfrak{p}))}{\mathcal{N}_{K}(\mathfrak{p})}&=\sum_{\mathfrak{p}\in P}\left( \sum_{s\in S'}1_{\mathfrak{p}\in P'_{\pi_{1}(s)}}\right)\dfrac{\log(\mathcal{N}_{K}(\mathfrak{p}))}{\mathcal{N}_{K}(\mathfrak{p})}
 \\ &
\leq\sum_{\mathfrak{p}\notin P^{\prime}}\beta|S'|\dfrac{\log(\mathcal{N}_{K}(p))}{\mathcal{N}_{K}(\mathfrak{p})}+\sum_{\mathfrak{p}\in P^{\prime}}|S'|\dfrac{\log(\mathcal{N}_{K}(\mathfrak{p}))}{\mathcal{N}_{K}(\mathfrak{p})} \\ & =\beta|S'|w(P\backslash P^{\prime})+|S'|w(P^{\prime}) \\ &= \beta|S'|w(P)-\beta|S'|w(P^{\prime})+|S'|w(P^{\prime}).
\end{split}
\end{equation*}
On the other hand, recalling that $w(P'_{x})\geq \kappa_{1}w(P_{x})\geq \dfrac{\kappa_{1}}{2}w(P)$, we have
$$
\sum_{s\in S'}\sum_{\mathfrak{p}\in P}1_{\mathfrak{p}\in P'_{\pi_{1}(s)}}\dfrac{\log(\mathcal{N}_{K}(\mathfrak{p}))}{\mathcal{N}_{K}(\mathfrak{p})}=\sum_{s\in S'}w(P'_{\pi_{1}(s)})\geq \sum_{s\in S'}\dfrac{\kappa_{1}}{2}w(P)= \dfrac{\kappa_{1}}{2}w(P)|S'|.$$
We conclude that
$\left(\dfrac{\kappa_{1}}{2}-\beta\right)w(P)\leq \left(1-\beta\right)w(P^{\prime})<w(P^{\prime})$.
It is enough to set 
\begin{equation}\label{valor de beta}
\beta=\dfrac{\kappa_{1}}{4}.
\end{equation}
\end{proof}
To finish the proof of Lemma \ref{lemma 3.4}, for each $\mathfrak{p}\in P^{\prime}$ we construct a large generic set modulo $\mathfrak p$, $\mathcal{G}_{\mathfrak{p}}(S)$.
 For that, take 
\begin{equation} 
\mathcal{G}_{\mathfrak{p}}(S):=\displaystyle \bigcupdot_{x:\mathfrak{p}\in P'_{x}}\mathcal{G}_{\mathfrak{p}}(S'_{x}).
\end{equation}
Observe that $\mathcal{G}_{\mathfrak{p}}(S)\cap \pi_{1}^{-1}(x)=\mathcal{G}_{\mathfrak{p}}(S'_{x})$ is a $(B,d-h-1)$-generic set modulo $\mathfrak{p}$ for all $x\in \pi_{1}(\mathcal{G}_{\mathfrak{p}}(S))$. 
The inequality
\begin{align}\label{c1}
|\mathcal{G}_{\mathfrak{p}}(S)|&=\left| \bigcupdot_{x:\mathfrak{p}\in P'_{x}}\mathcal{G}_{\mathfrak{p}}(S'_{x})\right| =\sum_{x:\mathfrak{p}\in P'_{x}}\left|\mathcal{G}_{\mathfrak{p}}(S'_{x})\right|\geq c_{1}\sum_{x:\mathfrak{p}\in P'_{x}}\left|S'_{x}\right|\\
&=c_{1}\left| \left\lbrace s \in S' : \mathfrak{p} \in P'_{\pi_{1}(s)} \right\rbrace \right| > c_{1}\beta|S'|\geq \dfrac{c_{1}\beta}{2^{2}}|S|\nonumber
\end{align}
shows that $ \mathcal{G}_{\mathfrak{p}}(S)$ is large in $S$.

Let us now see that $\mathcal{G}_{\mathfrak{p}}(S)$ is indeed a generic set modulo $\mathfrak{p}$. First, note that by construction, the residue classes modulo $\mathfrak{p}$ of $\pi_{1}(\mathcal{G}_{\mathfrak{p}}(S))$ do not lie in $\mathcal{E}(\mathfrak{p})$. Then, recalling the definition of $\mathcal{E}_{2}(\mathfrak{p}) \subseteq\mathcal{E}(\mathfrak{p})$ given in \eqref{definicion de e2}, it follows that
\begin{equation}\label{algo}
|\mathcal{G}_{\mathfrak{p}}(S)(a,\mathfrak{p})|\leq \dfrac{B_{1}}{\alpha \mathcal{N}_{K}(\mathfrak{p})}|S|\leq \dfrac{2^{2}B_{1}}{c_{1}\beta \alpha}\dfrac{|\mathcal{G}_{\mathfrak{p}}(S)|}{\mathcal{N}_{K}(\mathfrak{p})},
\end{equation}
where the second inequality is due to   \eqref{c1}.
Let ${\bf a}$ be a residue class modulo  $\mathfrak{p}$. Observe that
\begin{align}
\mathcal{G}_{\mathfrak{p}}(S)({\bf a},\mathfrak{p})&=\left( \bigcupdot_{x:\mathfrak{p}\in P'_{x}}\mathcal{G}_{\mathfrak{p}}(S'_{x})\right)({\bf a},\mathfrak{p})\\
&=\bigcupdot_{\substack{x:\mathfrak{p}\in P'_{x}}}\mathcal{G}_{\mathfrak{p}}(S'_{x})({\mathbf a},\mathfrak{p})\nonumber\\
&=\bigcupdot_{\substack{x:\mathfrak{p}\in P'_{x}\\x\equiv \pi_{1}({\bf a})(\text{mod}\ \mathfrak{p})}}\mathcal{G}_{\mathfrak{p}}(S'_{x})({\bf a},\mathfrak{p}).\nonumber
\end{align}
Hence the fact that $\mathcal{G}_{\mathfrak{p}}(S'_{x})$ is $(B,d-h-1)$-generic modulo $\mathfrak{p}$  implies
\begin{align}\label{4.24}
\displaystyle |\mathcal{G}_{\mathfrak{p}}(S)({\bf a},\mathfrak{p})| & =\left|\bigcupdot_{\substack{x:\mathfrak{p}\in P'_{x}\\x\equiv \pi_{1}({\bf a})(\text{mod}\ \mathfrak{p})}}\mathcal{G}_{\mathfrak{p}}(S'_{x})({\bf a},\mathfrak{p})\right|\\
&= \sum_{\substack{x:\mathfrak{p}\in P'_{x}\\x\equiv \pi_{1}({\bf a})(\text{mod}\ \mathfrak{p})}}\left| \mathcal{G}_{\mathfrak{p}}(S'_{x})({\bf a},\mathfrak{p})\right| \nonumber\\
&\leq \sum_{\substack{x:\mathfrak{p}\in P'_{x}\\x\equiv \pi_{1}({\bf a})(\text{mod}\ \mathfrak{p})}}|\mathcal{G}_{\mathfrak{p}}(S'_{x})|\dfrac{B}{\mathcal{N}_{K}(\mathfrak{p})^{d-h-1}}  \nonumber\\
&= \dfrac{B}{\mathcal{N}_{K}(\mathfrak{p})^{d-h-1}}\sum_{\substack{x:\mathfrak{p}\in P'_{x}\\x\equiv \pi_{1}({\bf a})(\text{mod}\ \mathfrak{p})}}|\mathcal{G}_{\mathfrak{p}}(S'_{x})|\nonumber\\
&=\dfrac{B}{\mathcal{N}_{K}(\mathfrak{p})^{d-h-1}}\left| \bigcupdot_{\substack{x:\mathfrak{p}\in P'_{x}\\x\equiv \pi_{1}({\bf a})(\text{mod}\ \mathfrak{p})}}\mathcal{G}_{\mathfrak{p}}(S'_{x}) \right|.\nonumber
\end{align}

The fact that $\bigcupdot_{x:x\equiv \pi_{1}({\bf a})(\text{mod}\ \mathfrak{p})}\mathcal{G}_{\mathfrak{p}}(S'_{x})\subseteq \mathcal{G}_{\mathfrak{p}}(S)(\pi_{1}({\bf a}),\mathfrak{p})$, 
combined with \eqref{4.24} and \eqref{algo} gives that
\begin{align}\label{B}
\displaystyle |\mathcal{G}_{\mathfrak{p}}(S)({\bf a},\mathfrak{p})|&\leq \dfrac{B}{\mathcal{N}_{K}(\mathfrak{p})^{d-h-1}}\left| \mathcal{G}_{\mathfrak{p}}(S)(\pi_{1}({\bf a}),\mathfrak{p}) \right|\\
&\leq \dfrac{B}{\mathcal{N}_{K}(\mathfrak{p})^{d-h-1}}2^{2}\dfrac{B_{1}}{c_{1}\beta\alpha}\dfrac{1}{\mathcal{N}_{K}(\mathfrak{p})}|\mathcal{G}_{\mathfrak{p}}(S)|\nonumber\\
&=2^{2}\dfrac{BB_{1}}{c_{1}\beta\alpha}\dfrac{1}{\mathcal{N}_{K}(\mathfrak{p})^{d-h}}|\mathcal{G}_{\mathfrak{p}}(S)|.\nonumber
\end{align}
\end{proof}
The proof also shows that the constants in  Lemma \ref{lemma 3.4} can be inductively computed as follows: from Claim \ref{claim 4.6},
\begin{equation}
\kappa_{1}(d,h,\varepsilon,\kappa,\alpha,c,K)=\dfrac{1}{4}\kappa_{1}\left(d-1,h,\nu\varepsilon,\frac{\kappa}{2},\frac{2\alpha}{C_{1}},\frac{c'(K,\nu)}{2^{d}},K\right)
\label{kappa1f}
\end{equation}
where $2^{d}$ is due to the assumption made at \eqref{3.4};
from \eqref{valor de beta} and \eqref{c1} we have that $c_{1}(d,h,\varepsilon,\kappa,\alpha,c,K)$ is equal to
\begin{equation}
\dfrac{\kappa_{1}\left(d-1,h,\nu\varepsilon,\frac{\kappa}{2},\frac{2\alpha}{C_{1}},\frac{c'(K,\nu)}{2^{d}},K\right)
 c_{1}\left(d-1,h,\nu\varepsilon,\frac{\kappa}{2},\frac{2\alpha}{C_{1}},\frac{c'(K,\nu)}{2^{d}},K\right)}{2^{d+4}}
\label{c1f}
\end{equation}
where $2^{d}$ is due to the assumption made at \eqref{3.4}; 
and from \eqref{C1} and \eqref{B} we have that $B(d,h,\varepsilon,\kappa,\alpha,c,K)$ is equal to
\begin{equation}
\dfrac{2^{12}dc_{5,K}B\left(d-1,h,\nu\varepsilon,\frac{\kappa}{2},\frac{2\alpha}{C_{1}},\frac{c'(K,\nu)}{2^{d}},K\right)}{\kappa\varepsilon\kappa_{1}\left(d-1,h,\nu\varepsilon,\frac{\kappa}{2},\frac{2\alpha}{C_{1}},\frac{c'(K,\nu)}{2^{d}},K\right)c_{1}\left(d-1,h,\nu\varepsilon,\frac{\kappa}{2},\frac{2\alpha}{C_{1}},\frac{c'(K,\nu)}{2^{d}},K\right)}
\label{Bf}
\end{equation}
where, by \eqref{C1} and \eqref{C1original}
\begin{equation}
C_{1}=C_{1}\left( \dfrac{\kappa}{4},\dfrac{1}{4},\dfrac{\varepsilon}{2d},K\right)=\dfrac{\kappa\varepsilon}{2^{7}c_{5,K}d}.
\label{C1f}
\end{equation}

\subsection{Characteristic sets}

Having proved the existence of generic subsets, now we prove that there exist ``small characteristic subsets''. In order to make precise the notion of characteristic subset, first we define what we mean by a ``small'' polynomial.
\begin{definition}[$r$-polynomial]\label{def of r-poly}
Let $K$ be a global field of degree $d_{K}$. Given a parameter $N$ and some integer $d>0$, by an $r$-polynomial we mean a non-zero polynomial $f\in \mathcal{O}_{K}[X_{1},\ldots ,X_{d}]$ such that for all ${\bf x}\in [N]_{\mathcal{O}_{K}}^{d}$ it holds $H_{K}(f({\bf x}))<N^{3rd_{K}}$.  
\end{definition}
We remark that the reason of the exponent $3rd_{K}$ in the definition is because for any $f\in \mathcal{O}_{K}[X_{1},\ldots ,X_{d}]$ with degree bounded by $r$ and coefficients of height bounded by $N^{rd_{K}}$ satisfies that for $N$ large enough, $H_{K}(f({\bf x}))<N^{3rd_{K}}$ for all $[N]_{\mathcal{O}_{K}}^{d}$. 

\begin{definition}[Characteristic subset]
Let $0<\delta\leq 1$ be a real number and $r>0$ a positive integer. We say that $A\subseteq S$ is $(r,\delta)$-characteristic for $S$ if there exists $A\subseteq L\subseteq S$ of size $|L|\geq \delta|S|$ such that whenever an $r$-polynomial vanishes at $A$, it also vanishes at $L$.
\label{characteristic} 
\end{definition}

The main result of this section is that there exists a positive constant $\delta>0$ 
such that for any $S$ as in the statement of Theorem \ref{theorem 2.4},
there exists small $(r,\delta)$-characteristic subsets, provided that $N$ is large enough. 

\begin{proposition}[Compare to Proposition 2.2 in \cite{Walsh2}]
Let $d,h\geq 0$ be arbitrary integers and $\varepsilon>0$ a positive real number. Let $Q=N^{\frac{\varepsilon}{2d}}$ and $P\subseteq \mathcal{P}(Q)$ satisfying $w(P)\geq \kappa w(\mathcal{P}(Q))$ for some $\kappa>0$. Let $r>0$ be an integer. Suppose that $S\subseteq [N]_{\mathcal{O}_{K}}^{d}$ is a subset of size $|S|\geq cN^{d-h-1+\varepsilon}$ occupying at most $\alpha\mathcal{N}_{K}(\mathfrak{p})^{d-h}$ residue classes modulo $\mathfrak{p}$ for all $\mathfrak{p}\in P$ and some $\alpha>0$. Then, if $N$ is sufficiently large, there exists a set $A\subseteq S$ of size $|A|\leq c_{2}r^{d-h}$ which is $(r,\delta)$-characteristic for $S$, for some $\delta=\delta(d,h,\varepsilon,\kappa,\alpha,c,K)$ and $c_{2}=c_{2}(d,h,\varepsilon, \kappa,\alpha,c,K)$.
\label{proposition 2.2}
\end{proposition}

\begin{proof}[Proof of Proposition \ref{proposition 2.2}]
We fix $h$ and proceed by induction on $d$. When $d< h$, for $N$ sufficiently large there exists $\mathfrak{p} \in P$ such that $\alpha\mathcal{N}_{K}(\mathfrak{p})^{d-h}<1$, hence $S$ is empty. When $d=h$, the result follows from Lemma \ref{lemma 3.2}. Indeed, in this case any subset $S\subseteq [N]_{\mathcal{O}_{K}}^{d}$ satisfying the hypothesis of Proposition \ref{proposition 2.2} occupies at most $\alpha$ residue classes modulo $\mathfrak{p}$ for all $\mathfrak{p}\in P$ and some $\alpha>0$. Thus, Lemma \ref{lemma 3.2} implies that $|S|\leq C_{2}=C_{2}(\alpha,\kappa,\frac{\varepsilon}{2d},K)$. Taking $A=B=S$ we see that $A$ is $(r,1)$-characteristic for $S$.  In particular, we have that
\begin{equation}
c_{2}(d,h,\varepsilon,\kappa,\alpha,c,K)=C_{2},
\label{c20}
\end{equation}
\begin{equation}
\delta(d,h,\varepsilon,\kappa,\alpha,c,K)=1.
\label{delta0}
\end{equation}
Hence, let us assume that $d\geq h+1$ and that the result holds for smaller dimensions. First we find generic subsets inside the sections of $S$ for many primes $\mathfrak{p}$, as we did in Lemma \ref{lemma 3.4}. Proceeding as in Claim \ref{claim1}, we pass to a subset $S_{1}\subseteq S$ of size $|S_{1}|\geq \frac{1}{8^{d}}|S|$ such that
\begin{equation}
|\pi_{1}(A)|\geq Q\text{ for all }A\subseteq S_{1}\text{ with }|A|\geq \dfrac{|S_{1}|}{8}. 
\label{4.1}
\end{equation}
We may further refine the set $S_{1}$ to have a control in the size of the sections:
\begin{claim}
There exists a subset $S_{2}\subseteq S_{1}$ of size $|S_{2}|\geq \frac{3}{4}|S_{1}|$, such that
\begin{equation}
|(S_{2})_{x}|\leq \dfrac{|S|}{Q}\text{ for all }x\in [N]_{\mathcal{O}_{K}},
\label{4.2}
\end{equation}
where $(S_{2})_{x}:=\pi_{1}^{-1}(x)\cap S_{2}$.
\label{claim5}
\end{claim}
\begin{proof}[Proof of Claim \ref{claim5}]
Let
$W:=\left\{ s\in S_{1}:\left |(S_{1})_{\pi_{1}(s)}\right |>\dfrac{|S_{1}|}{Q}\right\}.$
Note that $|W|\leq \frac{1}{4}|S_{1}|$, otherwise $|W|>\frac{1}{4}|S_{1}|>\frac{1}{8}|S_{1}|$ and \eqref{4.1} would imply $|\pi_{1}(W)|\geq Q$. This  would entail that
$$|S_{1}|=\left | \displaystyle \bigcup_{x\in \pi_{1}(S_{1})}(S_{1})_{x} \right|\geq \left| \bigcup_{x\in \pi_{1}(W)}(S_{1})_{x} \right|>|\pi_{1}(W)|\dfrac{|S_{1}|}{Q}\geq |S_{1}|,$$
which is a contradiction. Define $S_{2}:=S_{1}\backslash W$. Then $|S_{2}|\geq \frac{3}{4}|S_{1}|$ and for any $x\in \pi_{1}(S_{2})$, $(S_{2})_{x}=(S_{1})_{x}$, thus for such $x$ it holds
$$\left|(S_{2})_{x}\right|=\left|(S_{1})_{x}\right|\leq \dfrac{|S_{1}|}{Q}\leq \dfrac{|S|}{Q}.$$
\end{proof}
Now, we proceed as in the proof Claim \ref{claim3} to obtain a subset $S_{3}\subseteq S_{2}$ of size $|S_{3}|\geq \frac{1}{2}|S_{2}|$ such that 
\begin{equation}
|(S_{3})_{x}|\geq c'(K,\nu)N^{d-h-2+\nu\varepsilon}\text{ for all }x\in \pi_{1}(S_{3})
\label{4.25}
\end{equation}
where $\nu=\frac{d-1}{d}$ and 
\begin{equation}\label{dependencia de c prima de ka nu dos}
c'(K,\nu)=\dfrac{3c}{2^{3d+4}c''(K)}    \left( \dfrac{(1-\nu)\varepsilon}{d_{K}} \right)^{d_{K}}.
\end{equation}

Let $B_{1}$ be a large constant. For any prime $\mathfrak{p}$ we denote $\mathcal{E}_{1}(\mathfrak{p})$ for the set of residue classes $a\in \mathcal{O}_{K}/\mathfrak{p}$ such that $|[S_{3}(a,\mathfrak{p})]_{\mathfrak{p}}|\geq B_{1}\mathcal{N}_{K}(\mathfrak{p})^{d-h-1}$. Since $S_3\subseteq S$ 
and $\left |[S]_{\mathfrak{p}}\right |\leq \alpha\mathcal{N}_{K}(\mathfrak{p})^{d-h}$,
 it follows that $|\mathcal{E}_{1}(\mathfrak{p})|\leq \frac{\alpha}{B_{1}}\mathcal{N}_{K}(\mathfrak{p})$. Applying Lemma \ref{lemma 3.1} as in the proof of Lemma \ref{lemma 3.4} we conclude that if $B_{1}$ is sufficiently large enough, namely, if
\begin{equation}
B_{1}:=\dfrac{\alpha}{C_{1}\left(\dfrac{\kappa}{4},\dfrac{1}{4},\dfrac{\varepsilon}{2d},K\right)}=\frac{2^{7}\alpha c_{5,K}d}{\kappa\varepsilon},
\label{C1.2}
\end{equation}
then the set
$$X:=\left\{ x\in [N]_{\mathcal{O}_{K}}:\sum_{\mathfrak{p}\in P}1_{x(\text{mod}\ \mathfrak{p})\in \mathcal{E}_{1}(\mathfrak{p})}\dfrac{\log(\mathcal{N}_{K}(\mathfrak{p}))}{\mathcal{N}_{K}(\mathfrak{p})}\geq \dfrac{1}{2}w(P) \right\},$$
has size $|X|<Q$. Thus, $|\pi_{1}^{-1}(X)\cap S_{3}|<\frac{1}{8}|S_{1}|$, since otherwise the inequality $|\pi_{1}^{-1}(X)\cap S_{3}|\geq \frac{1}{8}|S_{1}|$ and \eqref{4.1} imply $|\pi_{1}(\pi_{1}^{-1}(X)\cap S_{3})|\geq Q$. But this can not hold, since $|\pi_{1}(\pi_{1}^{-1}(X)\cap S_{3})|\leq |X|<Q$. Since $|S_{3}|\geq \frac{1}{2}|S_{2}|\geq \frac{3}{8}|S_{1}|$, we conclude that $|S_{3}\backslash \pi_{1}^{-1}(X)|\geq \frac{1}{4}|S_{1}|$. 

If we set $S_{4}:=S_{3}\backslash \pi_{1}^{-1}(X)$, its nonempty sections coincide with the nonempty sections of $S_3$, hence by
 \eqref{4.25}, we have that
 \begin{equation}
|(S_{4})_{x}|=|(S_{3})_{x}|\geq c'(K,\nu)N^{d-h-2+\nu\varepsilon}\text{ for all }x\in \pi_{1}(S_{4}).
\end{equation}
Moreover, for all $x\in \pi_{1}(S_{4})$ we have a subset of primes
$$P_{x}:=\left\{ \mathfrak{p}\in P:x(\text{mod}\ \mathfrak{p})\notin \mathcal{E}_{1}(\mathfrak{p}) \right\}$$
for which $w(P_{x})\geq \frac{1}{2}w(P)\geq \frac{1}{2}\kappa w(\mathcal{P}(Q))$, and $\left |[(S_4)_x]_\mathfrak{p}\right |\leq B_{1}\mathcal{N}_{K}(\mathfrak{p})^{d-h-1}$, for all $\mathfrak{p}\in P_{x}$.
Hence, we are in condition to apply the inductive hypothesis to the sections of  $S_4$ with the parameter $Q^\prime:=N^{\frac{\nu\varepsilon}{2(d-1)}}$, 
which is equal to $Q$ by our choice of $\nu$.
We may deduce that for each nonempty section $(S_{4})_{x}$ there exist $\delta_{0}$ and $c_{2}$, independent on $x$, such that $(S_{4})_{x}$ admits an $(r,\delta_{0})$-characteristic subset $A({x})$ of size $|A({x})|\leq c_{2}r^{d-h-1}$. The dependence of the constants is as follows,
\begin{equation}\label{dependencia de c2}
c_{2}=c_{2}\left(d-1,h,\nu\varepsilon,\dfrac{\kappa}{2},B_{1},c'(K,\nu),K\right),
\end{equation}
\begin{equation}\label{dependencia de delta}
\delta_{0}=\delta\left(d-1,h,\nu\varepsilon,\dfrac{\kappa}{2},B_{1},c'(K,\nu),K\right).
\end{equation}
Since each $A({x})$ is $(r,\delta_{0})$-characteristic for $(S_{4})_x$, there exists a ``witness''  subset $A({x})\subseteq S'_{x}\subseteq (S_{4})_{x}$  of size $|S'_{x}|\geq \delta_{0}|(S_{4})_{x}|$ satisfying the condition of Definition \ref{characteristic}. We define
 \begin{equation}
 S':=\bigcup_{x\in \pi_{1}(S_{4})}S'_{x}.\label{definicion de S prima}
 \end{equation}
(Note that the notation $S'_{x}$ is consistent with our previous use of subindices to denote sections, since for each $x\in \pi_{1}(S_{4})$ the section of $S'$ over $x$ is equal to the witness subset $S'_{x}$.).
Let us summarize the properties that $S'$ satisfies:  
\begin{equation}
|S'|\geq \delta_{0} |S_{4}|\geq\delta_{0}\dfrac{1}{4}|S_{1}|\geq\dfrac{\delta_{0}}{2^{3d+2}}|S|,\label{V'}
\end{equation}
\begin{equation}
|S'_{x}|\geq \delta_{0}|(S_{4})_{x}|\geq \delta_{0}c'(K,\nu)N^{d-h-2+\nu\varepsilon}, \text{\ for all\ } x\in \pi_{1}(S'),
\end{equation}
\begin{align}&\text{for all\ } x\in \pi_{1}(S'), A({x}) \text{ is a }(r,1)\text{-characteristic subset for }S'_{x} \text{ of size }\label{r1 caracteristico}\\
&|A({x})|\leq c_{2}r^{d-h-1}.\nonumber
\end{align}

Applying Lemma \ref{lemma 3.4} to each section $S'_x$, we can construct a subset of primes $P'_{x}\subseteq P_{x}$ with $w(P'_{x})\geq \kappa_{1}w(P_{x})$ such that for all $\mathfrak{p}\in P'_{x}$ there exists a $(B,d-h-1)$-generic subset $\mathcal{G}_{\mathfrak{p}}(S'_{x})\subseteq S'_{x}$, of size $|\mathcal{G}_{\mathfrak{p}}(S'_{x})|\geq c_{1}|S'_{x}|$, where the constants are independent of $\mathfrak{p}$ and $x$. Specifically,
\begin{equation}\label{dependence of B}
B=B\left(d-1,h,\nu\varepsilon,\dfrac{\kappa}{2},B_{1},\delta_{0}c'(K,\nu),K\right),
\end{equation}
\begin{equation}
\kappa_{1}=\kappa_{1}\left(d-1,h,\nu\varepsilon,\dfrac{\kappa}{2},B_{1},\delta_{0}c'(K,\nu),K\right),
\end{equation}
\begin{equation}\label{ce1}
c_{1}=c_{1}\left(d-1,h,\nu\varepsilon,\dfrac{\kappa}{2},B_{1},\delta_{0}c'(K,\nu),K\right).
\end{equation}
Proceeding as in Claim \ref{claim 4.6}, there exist $\beta>0$ and subset of primes $P'\subseteq P$ such that 
\begin{equation}
w(P')\geq \frac{\kappa_{1}}{4}w(P),
\label{cotadewpprima}
\end{equation}
and for all $\mathfrak{p}\in P'$ there are at least $\beta|S'|$ elements $s\in S'$ for which $\mathfrak{p}\in P'_{\pi_{1}(s)}$.
Thus, for $\mathfrak{p}\in P^{\prime}$
$$\mathcal{G}_{\mathfrak{p}}:=\bigcupdot_{x:\mathfrak{p}\in P'_{x}}\mathcal{G}_{\mathfrak{p}}(S'_{x})$$
is a subset of $S$ of size 
\begin{equation}
|\mathcal{G}_{\mathfrak{p}}|\geq \beta c_{1}|S'|\geq \beta c_{1}\dfrac{\delta_{0}}{2^{3d+2}}|S|.
\label{cotaGpdeS} 
\end{equation}
By construction, each nonempty section $(\mathcal{G}_{\mathfrak{p}})_{x}$ equals to $\mathcal{G}_{\mathfrak{p}}(S'_{x})$, which is a $(B,d-h-1)$-generic set modulo $\mathfrak{p}$.

In order to prove Proposition \ref{proposition 2.2}, we are going to glue some of the characteristic subsets that we found on each section of $S'$. In order to  obtain a small characteristic subset of $S'$ by this procedure, we need to locate sections of $S'$ containing the residue class of many elements of $S'$ for many primes $\mathfrak{p}$. This is the content of the next lemma.

\begin{lemma}
There exists a subset $\mathcal{B}\subseteq S'$ with $|\mathcal{B}|\geq c_{3}|S^\prime|$, such that each nonempty section $\mathcal{B}_{x}$ of $\mathcal{B}$ is equal to $S'_x$ and there exists a subset of primes $P''_{x}\subseteq P'$ with $w(P''_{x})\geq \kappa_{3}w(P')$, such that for every $\mathfrak{p} \in P''_{x}$
$$\left| \left\{ s\in S':[s]_{\mathfrak{p}}\in [\mathcal{B}_{x}]_{\mathfrak{p}} \right\} \right|\geq c_{4}\dfrac{|S^\prime|}{\mathcal{N}_{K}(\mathfrak{p})}.$$ 
\label{lemma 4.1}
\end{lemma}
\begin{proof}[Proof of Lemma \ref{lemma 4.1}]
We begin by fixing a prime $\mathfrak{p}\in P'$ and consider some residue class 
$a\in [\pi_{1}(\mathcal{G}_{\mathfrak{p}})]_{\mathfrak{p}} \subseteq \mathcal{O}_{K}/\mathfrak{p}$.
 Since $\mathfrak{p}$ is fixed we are going to denote $\mathcal{G}_{\mathfrak{p}}(a)$ those elements in $\mathcal{G}_{\mathfrak{p}}$ with first coordinate congruent to $a$ modulo $\mathfrak{p}$. Moreover, given a class ${\bf b}\in (\mathcal{O}_{K}/\mathfrak{p})^{d}$ we denote $\mathcal{G}_{\mathfrak{p}}({\bf b})$ those elements of $\mathcal{G}_{\mathfrak{p}}$ congruent to ${\bf b}$ modulo $\mathfrak{p}$. By the pigeonhole principle and the fact that by construction of $P'$ and $\mathcal{G}_{\mathfrak{p}}$, it holds $|[\mathcal{G}_{\mathfrak{p}}(a)]_{\mathfrak{p}}|\leq B_{1}\mathcal{N}_{K}(\mathfrak{p})^{d-h-1}$ it follows that we can find ${\bf b}_{1}\in [\mathcal{G}_{\mathfrak{p}}(a)]_{\mathfrak{p}}\subseteq (\mathcal{O}_{K}/\mathfrak{p})^{d}$ with
$$|\mathcal{G}_{\mathfrak{p}}({\bf b}_{1})|\geq \dfrac{|\mathcal{G}_{\mathfrak{p}}(a)|}{B_{1}\mathcal{N}_{K}(\mathfrak{p})^{d-h-1}}.$$
Consider now $\mathcal{B}_{1}\subseteq \mathcal{G}_{\mathfrak{p}}(a)$ defined by
\begin{equation}
\mathcal{B}_{1}:=\displaystyle \bigcup_{s:[s]_{\mathfrak{p}}={\bf b}_{1}}(\mathcal{G}_{\mathfrak{p}})_{\pi_{1}(s)},
\label{4.33}
\end{equation}
that is, $\mathcal{B}_{1}$ is the union of those sections  of $\mathcal{G}_{\mathfrak{p}}$ containing a representative of ${\bf b}_{1}$. 

Since each $(\mathcal{G}_{\mathfrak{p}})_{x}$ is $(B,d-h-1)$-generic modulo $\mathfrak p$, we have 
$$|(\mathcal{G}_{\mathfrak{p}})_{x}|\geq \dfrac{\mathcal{N}_{K}(\mathfrak{p})^{d-h-1}}{B}|(\mathcal{G}_{\mathfrak{p}})_{x}({\bf b}_{1})|,$$
thus
\begin{equation}
|\mathcal{B}_{1}|\geq \dfrac{\mathcal{N}_{K}(\mathfrak{p})^{d-h-1}}{B}|\mathcal{G}_{\mathfrak{p}}({\bf b}_{1})|\geq \dfrac{1}{B_{1}B}|\mathcal{G}_{\mathfrak{p}}(a)|.
\label{4.4}
\end{equation}
Note that since $|\mathcal{G}_{\mathfrak{p}}(a)|\geq |\mathcal{B}_{1}|$ and $|[\mathcal{G}_{\mathfrak{p}}(a)]_{\mathfrak{p}}|\leq B_{1}\mathcal{N}_{K}(\mathfrak{p})^{d-h-1}$, by the first inequality of \eqref{4.4} and the pigeonhole principle we may find another residue class ${\bf b}_{2}\in [\mathcal{G}_{\mathfrak{p}}(a)]_{\mathfrak{p}}$ with
\begin{align*}
|\mathcal{G}_{\mathfrak{p}}({\bf b}_{2})|&\geq \dfrac{1}{B_{1}\mathcal{N}_{K}(\mathfrak{p})^{d-h-1}}|\mathcal{G}_{\mathfrak{p}}(a)\backslash \mathcal{G}_{\mathfrak{p}}({\bf b}_{1})|\\
&\geq \dfrac{1}{B_{1}\mathcal{N}_{K}(\mathfrak{p})^{d-h-1}}\left( 1-\dfrac{B}{\mathcal{N}_{K}(\mathfrak{p})^{d-h-1}} \right)|\mathcal{G}_{\mathfrak{p}}(a)|,
\end{align*}
which is at least $\frac{|\mathcal{G}_{\mathfrak{p}}(a)|}{2B_{1}\mathcal{N}_{K}(\mathfrak{p})^{d-h-1}}$ if $\mathcal{N}_{K}(\mathfrak{p})^{d-h-1}>2B$. In such case, if we define $\mathcal{B}_{2}$ as in \eqref{4.33}, the same reasoning that gives \eqref{4.4}  implies $|\mathcal{B}_{2}|\geq \frac{1}{2B_{1}B}|\mathcal{G}_{\mathfrak{p}}(a)|$. Iterating this process we obtain a sequence ${\bf b}:=\{{\bf b}_{1},\ldots ,{\bf b}_{q}\}$ of residue classes, $q=\left\lceil \frac{\mathcal{N}_{K}(\mathfrak{p})^{d-h-1}}{2B}\right\rceil$, satisfying
\begin{align*}
|\mathcal{G}_{p}({\bf b}_{j})|&\geq \dfrac{1}{B_{1}\mathcal{N}(p)^{d-h-1}}\left |\mathcal{G}_{p}(a)\backslash \bigcup_{i=1}^{j-1}\mathcal{G}_{p}({\bf b}_{i})
\right |\\
&\geq \dfrac{1}{B_{1}\mathcal{N}(p)^{d-h-1}}\left(1-\dfrac{(q-1)B}{\mathcal{N}(p)^{d-h-1}}\right)|\mathcal{G}_{p}(a)|,
\end{align*}
and $|\mathcal{B}_{j}|\geq \frac{1}{2B_{1}B}|\mathcal{G}_{\mathfrak{p}}(a)|$. In particular,
\begin{equation}
\displaystyle \sum_{j=1}^{q}|\mathcal{B}_{j}|\geq \dfrac{q}{2B_{1}B}|\mathcal{G}_{\mathfrak{p}}(a)|.
\label{4.5}
\end{equation} 
Now, consider the set
$$\mathcal{B}[a]:=\left \{s\in \mathcal{G}_{\mathfrak{p}}(a):\sum_{j=1}^{q}1_{s\in \mathcal{B}_{j}}\geq \dfrac{q}{4B_{1}B}\right \}.$$

Note that $\mathcal{B}[a]_{x}=\mathcal{B}[a]\cap \pi_{1}^{-1}(x)$ equals to $(\mathcal{G}_{\mathfrak{p}})_{x}$ whenever this intersection is nonempty. Additionally, \eqref{4.5} implies 
\begin{align}
\dfrac{q}{2B_{1}B}|\mathcal{G}_{p}(a)|& \leq \displaystyle \sum_{j=1}^{q}|\mathcal{B}_{j}|=\sum_{j=1}^{q}\sum_{s}1_{s\in \mathcal{B}_{j}} \\ 
&=\sum_{s\in \mathcal{B}[a]}\sum_{j=1}^{q}1_{s\in \mathcal{B}_{j}}+\sum_{s\notin \mathcal{B}[a]}\sum_{j=1}^{q}1_{s\in \mathcal{B}_{j}}\nonumber \\  
& \leq \sum_{s\in \mathcal{B}[a]}q+\sum_{s\notin \mathcal{B}[a]}\dfrac{q}{4B_{1}B}\nonumber\\
&\leq q\left(|\mathcal{B}[a]|+\dfrac{1}{4B_{1}B}|\mathcal{G}_{p}(a)|\right),\nonumber
\end{align}
from which we deduce
\begin{equation}
|\mathcal{B}[a]|\geq \dfrac{1}{4B_{1}B}|\mathcal{G}_{\mathfrak{p}}(a)|.
\label{4.50}
\end{equation}
\begin{claim}\label{claim 513}For each $x$ such that $\mathcal{B}[a]_{x}\neq \emptyset$, there are at least $\frac{|\mathcal{G}_{\mathfrak{p}}(a)|}{(4B_{1}B)^{2}}$ elements $s\in \mathcal{G}_{\mathfrak{p}}(a)$ such that $s\equiv {\bf y}(\text{mod} \ \mathfrak{p})$ for some ${\bf y}\in \mathcal{B}[a]_{x}$.
\end{claim}
\begin{proof}[Proof of Claim \ref{claim 513}.]
 Let $x$ such that $\mathcal{B}[a]_{x}\neq \emptyset$.
 Thus, $\mathcal{B}[a]_{x}=(\mathcal{G}_{\mathfrak{p}})_{x}$ and by definition of $\mathcal{B}[a]$, it holds that $(\mathcal{G}_{\mathfrak{p}})_{x}\subseteq \mathcal{B}_{j}$ for at least $\frac{q}{4B_{1}B}$ values of $j$. Now, fix any such $j$. Then by definition of $\mathcal{B}_{j}$, there exists $s_{j}\in \mathcal{G}_{\mathfrak{p}}(a)$ such that $(\mathcal{G}_{\mathfrak{p}})_{x}=(\mathcal{G}_{\mathfrak{p}})_{\pi_{1}(s_{j})}$ and $s_{j}\equiv {\bf b}_{j}(\text{mod} \ \mathfrak{p})$. Note that $s_{j}\in (\mathcal{G}_{\mathfrak{p}})_{x}$. Hence, for any $s\in \mathcal{G}_{\mathfrak{p}}({\bf b}_{j})$, it holds $s\equiv {\bf b}_{j}\equiv s_{j}(\text{mod} \ \mathfrak{p})$. We deduce that there are $|\mathcal{G}_{\mathfrak{p}}({\bf b}_{j})|\geq \frac{|\mathcal{G}_{\mathfrak{p}}(a)|}{2B_{1}\mathcal{N}_{K}(\mathfrak{p})^{d-h-1}}$ elements $s\in \mathcal{G}_{\mathfrak{p}}(a)$ such that $s\equiv {\bf y}(\text{mod} \ \mathfrak{p})$ for some ${\bf y}\in \mathcal{G}_{\mathfrak{p}}(a)$. Now, the elements we constructed have residue class ${\bf b}_{j}$ modulo $\mathfrak{p}$. Since the residue classes ${\bf b}_{j}$ are all different, we conclude that there are at least $\frac{q}{4B_{1}B}\frac{|\mathcal{G}_{\mathfrak{p}}(a)|}{2B_{1}\mathcal{N}_{K}(\mathfrak{p})^{d-h-1}}\geq \frac{|\mathcal{G}_{\mathfrak{p}}(a)|}{(4B_{1}B)^{2}}$ elements $s\in \mathcal{G}_{\mathfrak{p}}(a)$ such that $s\equiv {\bf y}(\text{mod} \ \mathfrak{p})$ for some ${\bf y}\in \mathcal{B}[a]_{x}$.
\end{proof}
Now, let 
\begin{equation}\label {definicion de R}
\mathcal{R}:=\left\lbrace a\in [\pi_{1}(S^{\prime})]_{\mathfrak{p}}\subseteq \mathcal{O}_{K}/\mathfrak{p}:|\mathcal{G}_{\mathfrak{p}}(a)|\geq \dfrac{1}{2\mathcal{N}_{K}(\mathfrak{p})}|\mathcal{G}_{\mathfrak{p}}| \right\rbrace 
\end{equation}
and let us denote
$$\mathcal{B}[\mathfrak{p}]:=\left\{ s\in S':(S')_{\pi_{1}(s)}\cap \mathcal{B}[a]\neq \emptyset \text{ for some }a\in \mathcal{R} \right\}.$$
In other words, $\mathcal{B}[\mathfrak{p}]$ consists of those sections of $S'$ with non-trivial intersection with $\bigcup_{a\in \mathcal{R}}\mathcal{B}[a]$. In particular, since $\mathcal{B}[\mathfrak{p}]$ contains the disjoint union $\bigcup_{a\in \mathcal{R}}\mathcal{B}[a]$, from the definition of $\mathcal{R}$ we deduce:
\begin{claim}
The following bound holds
$$|\mathcal{B}[\mathfrak{p}]|\geq \dfrac{\beta c_{1}}{8B_{1}B}|S^\prime|.$$
\label{claim6}
\end{claim} 
\begin{proof}[Proof of Claim \ref{claim6}]
First note that by the pigeonhole principle, $\mathcal{R}\neq \emptyset$.
Now, if $\mathcal{R}=\{a_{1},\ldots ,a_{h}\}$,
$$|\mathcal{G}_{\mathfrak{p}}|=\sum_{a}|\mathcal{G}_{\mathfrak{p}}(a)|=\sum_{a\in \mathcal{R}}|\mathcal{G}_{\mathfrak{p}}(a)|+\sum_{a\notin \mathcal{R}}|\mathcal{G}_{\mathfrak{p}}(a)|<\sum_{a\in \mathcal{R}}|\mathcal{G}_{\mathfrak{p}}(a)|+\dfrac{\mathcal{N}_{K}(\mathfrak{p})-h}{2\mathcal{N}_{K}(\mathfrak{p})}|\mathcal{G}_{\mathfrak{p}}|,$$
thus
$$\sum_{a\in \mathcal{R}}|\mathcal{G}_{\mathfrak{p}}(a)|>\left( 1-\dfrac{\mathcal{N}_{K}(\mathfrak{p})-h}{2\mathcal{N}_{K}(\mathfrak{p})} \right)|\mathcal{G}_{\mathfrak{p}}|=\dfrac{\mathcal{N}_{K}(\mathfrak{p})+h}{2\mathcal{N}_{K}(\mathfrak{p})}|\mathcal{G}_{\mathfrak{p}}|>\dfrac{|\mathcal{G}_{\mathfrak{p}}|}{2}.$$
Combining this together with \eqref{cotaGpdeS} and \eqref{4.50} we get
$$|\mathcal{B}[\mathfrak{p}]|\geq \left| \bigcup_{a\in \mathcal{R}}\mathcal{B}[a] \right|=\sum_{a\in \mathcal{R}}|\mathcal{B}[a]|\geq \sum_{a\in \mathcal{R}}\dfrac{1}{4B_{1}B}|\mathcal{G}_{\mathfrak{p}}(a)|>\dfrac{1}{8B_{1}B}|\mathcal{G}_{\mathfrak{p}}|>\dfrac{\beta c_{1}}{8B_{1}B}|S^{\prime}|.$$\end{proof}

For an element $s\in S'$ denote $P''_{s}$ for the set of primes $\mathfrak{p}\in P'$ for which $s\in \mathcal{B}[\mathfrak{p}]$. 
\begin{claim}\label{claim 5.15}
\label{kappatres}
There exist constants $\kappa_{3}$ and $c_{3}$ such that the set
$$\mathcal{B}:=\left\{ s\in S':w(P''_{s})\geq \kappa_{3}w(P^{\prime})\right\}$$
satisfies $|\mathcal{B}|\geq c_{3}|S^{\prime}|$.
\end{claim}
\begin{proof}[Proof of Claim \ref{claim 5.15}.]On the one hand we have
$$\sum_{\mathfrak{p}\in P'}\sum_{s\in S'}1_{s\in \mathcal{B}[\mathfrak{p}]}\dfrac{\log(\mathcal{N}_{K}(\mathfrak{p}))}{\mathcal{N}_{K}(\mathfrak{p})}=\sum_{\mathfrak{p}\in P'}|\mathcal{B}[\mathfrak{p}]|\dfrac{\log(\mathcal{N}_{K}(\mathfrak{p}))}{\mathcal{N}_{K}(\mathfrak{p})}\geq \dfrac{\beta c_{1}}{8B_{1}B}|S^{\prime}|w(P').
$$
 On the other hand,
\begin{align*}
\sum_{s\in S'}\sum_{\mathfrak{p}\in P'}1_{s\in \mathcal{B}[\mathfrak{p}]}\dfrac{\log(\mathcal{N}_{K}(\mathfrak{p}))}{\mathcal{N}_{K}(\mathfrak{p})}& =\sum_{s\in S'}\sum_{\mathfrak{p}\in P''_{s}}\dfrac{\log(\mathcal{N}_{K}(\mathfrak{p}))}{\mathcal{N}_{K}(\mathfrak{p})}\\
&=\sum_{s\in S'}w(P''_{s})\\
&=\sum_{s\in \mathcal{B}}w(P''_{s})+\sum_{s\notin \mathcal{B}}w(P''_{s})\\
&\leq |\mathcal{B}|w(P')+\kappa_{3}|S'|w(P^{\prime}),
\end{align*}
because $P''_{s}\subseteq P'$. Hence 
$\left( \dfrac{\beta c_{1}}{8B_{1}B}-\kappa_{3}\right) |S^{\prime}|\leq |\mathcal{B}|.$
Thus, it is enough to take
\begin{equation}
\kappa_{3}=c_3:=\dfrac{\beta c_{1}}{16B_{1}B}.
\label{kappa3}
\end{equation}
\end{proof}

Now, let us check that $\mathcal{B}$ satisfies the condition of Lemma \ref{lemma 4.1}. Let $x$ be such that $\mathcal{B}_{x}\neq \emptyset$. 
Observe that if $s_1,s_2\in \pi^{-1}_{1}(x)$, then $P''_{s_1}=P''_{s_2}$. It follows that  $P''_x:=P''_{s}$ with $s\in \pi^{-1}_{1}(x)$ is well defined. 
From this we conclude that  each nonempty section $\mathcal{B}_{x}$ of $\mathcal{B}$ is equal to $S'_x$.
Let $\mathfrak{p}\in P''_{x}$,
by construction, $\mathcal{B}[a]_{x}\subseteq \mathcal{B}_{x}$ for some $a\in \mathcal{R}$, where $\mathcal{R}$ was defined in \eqref{definicion de R}. Since there are at least $\frac{|\mathcal{G}_{\mathfrak{p}}(a)|}{(4B_{1}B)^{2}}$ elements $s\in \mathcal{G}_{\mathfrak{p}}(a)$ such that $s\equiv {\bf y}(\text{mod} \ \mathfrak{p})$ for some ${\bf y}\in \mathcal{B}[a]_{x}$, we conclude
\begin{align*}
\left| \left\{ s\in S':[s]_{\mathfrak{p}}\in [\mathcal{B}_{x}]_{\mathfrak{p}} \right\} \right|&\geq \dfrac{|\mathcal{G}_{\mathfrak{p}}(a)|}{(4B_{1}B)^{2}}\\
&\geq \dfrac{1}{2^{5}(B_{1}B)^{2}}\dfrac{1}{\mathcal{N}_{K}(\mathfrak{p})}|\mathcal{G}_{\mathfrak{p}}|\\
&\geq \dfrac{\beta c_1}{2^{5}(B_{1}B)^{2}}
\dfrac{|S^{\prime}|}{\mathcal{N}_{K}(\mathfrak{p})}\\
&=c_{4}\dfrac{|S^{\prime}|}{\mathcal{N}_{K}(\mathfrak{p})},
\end{align*}
where the second inequality is true because $a\in\mathcal R$, the third inequality is true due to  \eqref{cotaGpdeS}, and
\begin{equation}
c_{4}:=\dfrac{\beta c_1}{2^{5}(B_{1}B)^{2}}.
\label{c4}
\end{equation}
This finishes the proof of Lemma \ref{lemma 4.1}.
\end{proof}
In order to conclude the proof of Proposition \ref{proposition 2.2} we are going to show that if an $r$-polynomial vanishes at the sections $\mathcal{B}_{x}$ for $\gtrsim_{r,d,h,\varepsilon,\kappa,K}1$ values of $x$, then it should vanish at a positive proportion of $S$. To that end, we will first find $m=O_{r,d,h,\varepsilon,\kappa,K}(1)$ large enough and $S'_{x_{1}},\ldots ,S'_{x_{m}}$, $m$ different sections of $S'$ having non-trivial intersection with $\mathcal{B}$, and such that for many 
$s\in S'$ there are many primes $\mathfrak{p}\in\mathcal{P}(Q)$ for which there exists $s_j\in S'_{x_{j}}$ for some $1\leq j\leq m$ such that 
$s_j\equiv s (\text{mod} \ \mathfrak{p})$.

Recall that any nonempty section  $\mathcal B_x$ of $\mathcal B$ is equal to  $S'_x$. Since $S'\subseteq S_4\subseteq S_3\subseteq S_2$, 
it follows by \eqref{4.2} that $|\mathcal{B}_x|=|S'_x|\leq \dfrac{|S|}{Q}$, for all $x\in\pi_1(\mathcal {B})$. 
On the other hand, by Lemma \ref{lemma 4.1} and the inequality \eqref{V'}, we get that
 $|\mathcal{B}|\geq c_3\dfrac{\delta_0}{2^{3d+2}}|S|$. It follows that there are at least $\dfrac{c_3\delta_0}{2^{3d+2}}Q$ nonempty sections of $\mathcal {B}$.

Denote $\mathcal{L}:=\{S'_{x_{1}},\ldots ,S'_{x_{m}}\}$  a choice of $m$ sections of $S'$ such that $S'_{x_{i}}\cap \mathcal{B}\neq \emptyset$ for all $i$. Let $P_{\mathcal{L}}$ be the set of primes $\mathfrak{p}$ in $\mathcal{P}(Q)$  for which there exist a pair of sections $S'_{x_{i}}\neq S'_{x_{j}}$ in $\mathcal{L}$ such that $[S'_{x_{i}}]_{\mathfrak{p}}\cap [S'_{x_{j}}]_{\mathfrak{p}}\neq \emptyset$. Given such a pair of sections, the fact that $[S'_{x_{i}}]_{\mathfrak{p}}\cap [S'_{x_{j}}]_{\mathfrak{p}}\neq \emptyset$ implies $x_{i}\equiv x_{j}(\text{mod} \ \mathfrak{p})$. Since $x_{i}\neq x_{j}$ and are both in $[N]_{\mathcal O_{K}}$, by \eqref{height2} and \eqref{norm}, it follows that  
$\mathcal{N}_{K}(x_{i}-x_{j})\leq N^{3}$ when $N\geq2^{d_K}$.
Then
\begin{equation*}\begin{split}
\displaystyle \sum_{\mathfrak{p}\in P_{\mathcal{L}}}\log(\mathcal{N}_{K}(\mathfrak{p}))&\leq\sum_{\{i,j\},i\neq j}\sum_{\mathfrak{p}|x_{i}-x_{j}}\log(\mathcal{N}_{K}(\mathfrak{p}))\\
&\leq \sum_{\{i,j\},i\neq j}\log(\mathcal{N}_{K}(x_{i}-x_{j}))\\
&\leq{m\choose 2}3\log(N).
\end{split}
\end{equation*}
By a standard application of partial summation, this implies that 
\begin{equation}w(P_{\mathcal{L}})\lesssim_{r,d,h,\varepsilon, \kappa,K}\log(\log(N)).
\label{4.8}
\end{equation}
Now, we consider in $S'$ the function
$$\Psi_{\mathcal{L}}(s):=\sum_{\mathfrak{p}\in \mathcal{P}(Q)}1_{\exists {\bf x}\in \mathcal{L}:s\equiv {\bf x}(\text{mod} \ \mathfrak{p})}\log(\mathcal{N}_{K}(\mathfrak{p})).$$
Note that $\Psi_{\mathcal{L}}(s)$ measures how much a residue class occupied by $s$ contains a representative in $\mathcal{L}$.  Keeping in mind that nonempty sections of $\mathcal B$ are nonempty sections of $S'$, from Lemma 
\ref{lemma 4.1}  we deduce that
\begin{align*}
\sum_{s\in S'}\Psi_{\mathcal{L}}(s)&\geq \sum_{i=1}^{m}\sum_{\mathfrak{p}\in P_{x_{i}}\backslash P_{\mathcal{L}}}\sum_{s\in S'}1_{\exists {\bf x}\in S'_{x_{i}}:s\equiv {\bf x}(\text{mod} \ \mathfrak{p})}\log(\mathcal{N}_{K}(\mathfrak{p}))\\
&\geq \sum_{i=1}^{m}\sum_{\mathfrak{p}\in P_{x_{i}}\backslash P_{\mathcal{L}}}c_{4}|S^{\prime}|\dfrac{\log(\mathcal{N}_{K}(\mathfrak{p}))}{\mathcal{N}_{K}(\mathfrak{p})}\\
&=c_{4}|S^{\prime}| \sum_{i=1}^{m}\left(w(P_{x_{i}})-w(P_{\mathcal{L}}) \right),
\end{align*}
where the first inequality is because for any prime  $\mathfrak{p}\in \bigcup_{i}P_{x_{i}}\backslash P_{\mathcal{L}}$ we have $[S'_{x_{i}}]_{\mathfrak{p}}\cap [S'_{x_{j}}]_{\mathfrak{p}}=\emptyset$ for any $i\neq j$, thus the conditions $\exists {\bf x}\in S'_{x_{i}}:s\equiv {\bf x}(\text{mod} \ \mathfrak{p})$ are pairwise disjoint. Since
$w(P_{x_{i}})\geq \kappa_{3}w(P')\geq \frac{\kappa_{3} \kappa_{1}}{4}w(P)\geq \frac{\kappa_{3} \kappa_{1}\kappa}{4}w(\mathcal{P}(Q))$, inequality \eqref{4.8} implies that
\begin{align}\label{any L}
\sum_{s\in S'}\Psi_{\mathcal{L}}(s)&\geq mc_{4}|S^{\prime}|\left(\frac{\kappa_{3} \kappa_{1}\kappa}{4}w(\mathcal{P}(Q))+O_{r,d,h,\varepsilon,\kappa,K}(\log(\log(N)))\right)\\
&\geq \dfrac{1}{2}mc_{4}|S^{\prime}|\frac{\kappa_{3} \kappa_{1}\kappa}{4}w(\mathcal{P}(Q))\nonumber\\
&\geq mc_{4}\dfrac{\delta_{0}}{2^{3d+2}}|S|\frac{\kappa_{3} \kappa_{1}\kappa}{8}w(\mathcal {P}(Q))\nonumber\\ 
&\geq c_{5}m|S|\log(N),\nonumber
\end{align}
for $N$ large enough. Here,
\begin{equation}\label{c6}
c_{5}=\frac{\kappa}{2^{3d+6}}c_{2,K}\dfrac{\varepsilon}{d}c_{4}\kappa_{1} \kappa_{3}\delta_{0}.
\end{equation}

We will now bound $\sum_{s\in S'}\Psi_{\mathcal{L}}(s)$ from above.
In order to achieve that, first note that for $s$ in $\mathcal{L}$, \eqref{4.2} implies that
\begin{align}\label{in L}
\displaystyle \sum_{s\in \mathcal{L}}\Psi_{\mathcal{L}}(s)&=\sum_{s\in \mathcal{L}}\sum_{\mathfrak{p}\in \mathcal{P}(Q)}1_{\exists {\bf x}\in \mathcal{L}:s\equiv {\bf x}(\text{mod} \ \mathfrak{p})}\log(\mathcal{N}_{K}(\mathfrak{p}))\\
&=\sum_{s\in \mathcal{L}}\sum_{\mathfrak{p}\in \mathcal{P}(Q)}\log(\mathcal{N}_{K}(\mathfrak{p}))\nonumber\\
&\leq c_{4,K}Q|\mathcal{L}|\nonumber\\
&\leq c_{4,K}m|S|.\nonumber
\end{align}
Also note that if $s\notin \mathcal{L}$, then 
\begin{align}\label{not in L11}
\Psi_{\mathcal{L}}(s)&=\displaystyle \sum_{\mathfrak{p}\in \mathcal{P}(Q)}1_{\exists {\bf x}\in \mathcal{L}:s\equiv {\bf x}(\text{mod} \ \mathfrak{p})}\log(\mathcal{N}_{K}(\mathfrak{p}))\\
&\leq \sum_{\mathfrak{p}\in \mathcal{P}(Q)}\sum_{i=1}^{m}1_{\exists {\bf x}\in S'_{x_{i}}:s\equiv {\bf x}(\text{mod} \ \mathfrak{p})}\log(\mathcal{N}_{K}(\mathfrak{p})) \nonumber\\ 
&=\sum_{i=1}^{m}\sum_{\mathfrak{p}\in \mathcal{P}(Q)}1_{\exists {\bf x}\in S'_{x_{i}}:s\equiv {\bf x}(\text{mod} \ \mathfrak{p})}\log(\mathcal{N}_{K}(\mathfrak{p})) \nonumber\\
&\leq \sum_{i=1}^{m}\sum_{\mathfrak{p}|\pi_{1}(s)-x_{i}}\log(\mathcal{N}_{K}(\mathfrak{p}))\nonumber\\
&\leq \sum_{i=1}^{m}\log(\mathcal{N}_{K}(\pi_{1}(s)-x_{i})).\nonumber
\end{align}
By \eqref{height2}, \eqref{norm} and the fact that $\pi_{1}(s),x_{i}\in [N]_{\mathcal{O}_{K}}$, we have $\mathcal{N}_{K}(\pi_{1}(s)-x_{i})\leq N^{3}$ when $N\geq2^{d_K}$, thus from \eqref{not in L11} we deduce that for $s\notin \mathcal{L}$ it holds
\begin{equation}
\Psi_{\mathcal{L}}(s)\leq 3m\log(N).
\label{not in L1} 
\end{equation}
Let $\delta_{1}:=\dfrac{c_{5}}{4}$, and suppose that the set 
$$L:=\{s\in S':\Psi_{\mathcal{L}}(s)\geq \gamma\}$$
has  size at most $\delta_{1}|S|$.
Then 
\begin{align}\label{cota superior de suma de Psi}
\sum_{s\in {S'}}\Psi_{\mathcal{L}}(s)&
=\sum_{\substack{s\in {S'}\\ \Psi_{\mathcal{L}}(s)<\gamma}}\Psi_{\mathcal{L}}(s)+
\sum_{\substack{s\in \mathcal{L}\\ \Psi_{\mathcal{L}}(s)\geq\gamma}}\Psi_{\mathcal{L}}(s)+
\sum_{\substack{s\notin \mathcal{L}\\ \Psi_{\mathcal{L}}(s)\geq\gamma}}\Psi_{\mathcal{L}}(s)
\\
&\leq \sum_{\substack{s\in {S'}\\ \Psi_{\mathcal{L}}(s)<\gamma}}\Psi_{\mathcal{L}}(s)
+\sum_{s\in \mathcal{L}}\Psi_{\mathcal{L}}(s)+
\sum_{\substack{{s\in S'}\\ \gamma\leq \Psi_{\mathcal{L}}(s)\leq 3m\log(N)}}\Psi_{\mathcal{L}}(s)
\nonumber\\
&<\gamma|S|+c_{4,K}m|S|+\delta_1|S|3m\log(N).\nonumber
\end{align}
If  we now set $\gamma:=3rd_{K}\log(N)$, combining \eqref{any L} with \eqref{cota superior de suma de Psi} we get that
$$4\delta_1m\log(N)<3rd_{K}\log(N)+c_{4,K}m+3\delta_1m\log(N),$$
and then
$$m\left(\delta_1-\frac{c_{4,K}}{\log N}\right )< 3rd_{K}.$$
So, if we take $m:=\dfrac{4rd_{K}}{\delta_1}$ we reach a contradiction when $N>\exp\left(\dfrac{4c_{4,K}}{\delta_{1}}\right)$.
We conclude that for $N$ sufficiently large, the set
$$L:=\{s\in S':\Psi_{\mathcal{L}}(s)\geq 3rd_{K}\log(N)\}$$
has size $|L|\geq \delta_{1}|S|$ for our choices of $m$ and $\delta_{1}$.
Since 
$$\Psi_{\mathcal{L}}(s)=\log\left( \prod_{\substack{\mathfrak{p}\in \mathcal{P}(Q)\\ \exists {\bf x}\in \mathcal{L}:s\equiv {\bf x}(\text{mod} \ \mathfrak{p})}} \mathcal{N}_{K}(\mathfrak{p}) \right),$$
it follows that if $s\in L $ then $$\prod_{\substack{\mathfrak{p}\in \mathcal{P}(Q)\\ \exists {\bf x}\in \mathcal{L}:s\equiv {\bf x}(\text{mod} \ \mathfrak{p})}} \mathcal{N}_{K}(\mathfrak{p}) \geq N^{3rd_{K}}.$$

Now, let us see that if an $r$-polynomial vanishes at $\mathcal{L}$, then it must vanish at $L$. Indeed, let $f$ be such a polynomial and let 
$s\in L$. 
If $\mathfrak{p}\in\mathcal{P}(Q)$ is a prime such that there exists ${\bf x}\in \mathcal{L}$ with $s\equiv {\bf x} (\text{mod} \ \mathfrak{p})$, the fact that $f({\bf x})=0$ implies that $\mathfrak{p}|f(s)$. By Definition \ref{def of r-poly}, $H_{K}(f({\bf y}))<N^{3rd_{K}}$ for all ${\bf y}\in [N]_{\mathcal{O}_{K}}^{d}$. Hence if $f(s)\neq 0$ it holds that
$$N^{3rd_{K}}\leq \prod_{\substack{\mathfrak{p}\in \mathcal{P}(Q)\\ \exists {\bf x}\in \mathcal{L}:s\equiv {\bf x}(\text{mod} \ \mathfrak{p})}} \mathcal{N}(\mathfrak{p}) \leq \prod_{\mathfrak{p}|f(s)}\mathcal{N}_{K}(\mathfrak{p})\leq \mathcal{N}_{K}(f(s))\leq H_{K}(f(s))<N^{3rd_{K}},$$
which is absurd. Thus $f(s)=0$.

By the construction of $S'$ given in \eqref{definicion de S prima} and by \eqref{r1 caracteristico}, we know that each section $S'_{x_{i}}\in \mathcal{L}$ contains an $(r,1)$-characteristic subset $A(x_i)$, of size at most $c_{2}r^{d-h-1}$. Taking $A$ equal the union of these $m$ subsets we obtain an $(r,\delta_{1})$-characteristic subset for $S$, of size at most $\frac{m}{r}c_{2}r^{d-h}=\frac{4d_{K}}{\delta_{1}}c_{2}r^{d-h}$.
\end{proof}

\subsection{Construction of a polynomial of low complexity and conclusion of the proof}
\label{4.3}

Having constructed a characteristic subset $A\subseteq S$, the last step to prove Theorem \ref{theorem 2.4} is to construct a small polynomial vanishing at $A$. This will be done by using the following variant of Siegel's lemma.

\begin{lemma}[Lemma 3.6 in {\cite{P}}]
Let $K$ be a global field of degree $d_{K}$. Let $(a_{ij})_{i,j}$, $1\leq i\leq s$, $1\leq j\leq t$ be elements of $\mathcal{O}_{K}$ with $H(a_{ij})\leq C$ for all $i,j$. Let us suppose that $t>2d_{K}^{2}s$. Then, there exists ${\bf c}=(c_{1},\ldots ,c_{t})\in \mathcal{O}_{K}^{t}\backslash \{0\}$, such that
$$H_{K}(1:{\bf c})\leq c_{6,K}\left( tC\right)^{\frac{8d_{K}^{2}s}{t-2d_{K}s}}$$
and
$$\sum_{j=1}^{t}c_{j}a_{ij}=0\text{ for all }1\leq i\leq s.$$  
\label{siegel}
\end{lemma}
\begin{proof}[Proof of Theorem \ref{theorem 2.4}]
The proof of the theorem is exactly the same as the proof of Theorem 2.4 in \cite{Walsh2}, replacing the classical Siegel's lemma and Proposition 2.2 in \cite{Walsh2}  by Lemma \ref{siegel} and Proposition \ref{proposition 2.2}. We sketch the details.  First, it is enough to prove Theorem \ref{theorem 2.4} for $\eta=1-\delta$; the general case follows by a simple partitioning argument. Thus, we have 
$S\subseteq [N]_{\mathcal{O}_{K}}^{d}$  of size $|S|\geq cN^{d-h-1+\varepsilon}$ occupying at most $\alpha \mathcal{N}_{K}(\mathfrak{p})^{d-h}$ residue classes modulo $\mathfrak{p}$ for every prime $\mathfrak{p}\in P$.
By Proposition \ref{proposition 2.2} there exists a subset $A\subseteq S$ of size $|A|\leq c_{2}r^{d-h}$ which is $(r,\delta)$-characteristic for $S$, provided that $N$ is large enough. Recall that the number of monomials in $d$ variables of degree equal to $r$ is equal to ${r+d-1\choose d-1}$. Consider the system of $|A|$-linear equations in ${r+d-1\choose d-1}$ variables given by
$$\sum_{{\bf i}=(i_{1},\ldots, i_{d})}\beta_{{\bf i}}{\bf a}^{{\bf i}}=0\text{ for all }{\bf a}\in A,$$
where we are using the multi-index notation ${\bf a}^{{\bf i}}=a_{1}^{i_{1}}\ldots a_{d}^{i_{d}}$ for ${\bf a}=(a_{1},\ldots ,a_{d})$ and we are summing over the ${\bf i}$'s with $i_{1}+\ldots +i_{d}=r$. Note that $H_{K}({\bf a}^{{\bf i}})\leq N^{r}$. Now choose $r$ large enough such that
\begin{equation}
{r+d-1\choose d-1}>18d_{K}^{2}|A|,
\label{pie}
\end{equation}
namely, since ${r+d-1\choose d-1}\geq \frac{r^{d-1}}{(d-1)!}$ and $|A|\leq c_{2}r^{d-h}$ it is enough to choose
\begin{equation}
\dfrac{r^{d-1}}{(d-1)!}>18d_{K}^{2}c_{2}r^{d-h},\text{ i.e. }r>\left( 18d_{K}^{2}c_{2}(d-1)!\right)^{\frac{1}{h-1}}.
\end{equation}
\begin{remark*}
At this stage is the only place when we require that $h>1$.
\end{remark*}
 Note that \eqref{pie} implies
\begin{equation}
\dfrac{16d_{K}|A|}{{r+d-1\choose d-1}-2d_{K}|A|}<1.
\label{pie2}
\end{equation}
By Lemma \ref{siegel} and \eqref{pie2} there exists a solution $\left(\beta_{{\bf i}}\right)\in \mathcal{O}_{K}^{{r+d-1\choose d-1}}\backslash \{0\}$ satisfying
$$H_{K}(1:(\beta_{{\bf i}}))\leq c_{6,K}\left[{r+d-1\choose d-1}N^{r}\right]^{\frac{8d_{K}^{2}|A|}{{r+d-1\choose d-1}-2d_{K}|A|}}\leq c_{6,K}{r+d-1\choose d-1}^{\frac{d_{K}}{2}}N^{\frac{rd_{K}}{2}}\leq N^{rd_{K}}$$
provided $N$ is large enough such that ${r+d-1\choose d-1}\leq N^{r}$. Consider now the non-zero homogeneous polynomial $f=\sum_{{\bf i}}\beta_{{\bf i}}{\bf x}^{{\bf i}}$ of degree $r$. For any ${\bf x}\in [N]_{\mathcal{O}_{K}}^{d}$ \eqref{polynomialheight2} and \eqref{height1} imply that
$$H_{K}(f({\bf x}))\leq {r+d-1\choose d-1}^{d_{K}}H_{K}(1:(\beta_{{\bf i}}))H_{K}(1:{\bf x})^{r}<{r+d-1\choose d-1}^{d_{K}}N^{2rd_{K}}<N^{3rd_{K}}$$
for $N$ large enough. Thus, by definition \ref{def of r-poly}, $f$ is a non-zero homogeneous $r$-polynomial of degree $r$ that vanishes at $A$. This concludes the proof. In particular, we can choose
$r=\left\lceil\left( 18c_{2}d_{K}^{2}(d-1)! \right)^{\frac{1}{h-1}}\right\rceil$.
\end{proof}

\begin{remark}
Let us note that while Theorem \ref{theorem 2.4} implies Theorem \ref{reduction2}, and hence Theorem \ref{basis3}, it does not imply Theorem \ref{basis2}. One reason is that Theorem \ref{theorem 2.4} holds for $0\leq k<d-1$, while Theorem \ref{basis2} is valid for $0\leq k<d$. Another reason is that the statement of Theorem \ref{basis2} is in terms of sets $S$ lying in $[N]_{\mathcal{O}_{K}}^{d}$. Instead, if in Theorem \ref{basis2} the set $S$ lies in $[N]_{\mathbb{A}^{d}(\mathcal{O}_{K})}$ and occupies at most $\alpha\mathcal{N}_{K}(\mathfrak{p})^{k}$ residue classes for every prime $\mathfrak{p}$, with $0\leq k<d-1$, then the conclusion of the theorem follows immediately from Theorem \ref{reduction2}. In order to prove Theorem \ref{basis2} as it is stated in the introduction,  let us observe the basic fact that the number of monomials of degree at most $r$ in $d$ variables is equal to ${r+d\choose d}$. Keeping in mind that Lemma \ref{lemma 3.4} and Proposition \ref{proposition 2.2} are valid for any $h\geq1$,   a very easy modification in the construction of the auxiliary polynomial presented in this section gives the following result.
\end{remark}
 \begin{theorem}
Let $d,h$ be positive integers and let $\varepsilon,\eta>0$ be positive real numbers. Let $K$ be a global field of degree $d_{K}$ and $\mathcal{O}_{K}$ be its ring of integers. Set $Q=N^{\frac{\varepsilon}{2d}}$ and let $P\subseteq \mathcal{P}(Q)$ satisfy $w(P)\geq \kappa\log(Q)$ for some $\kappa>0$. Suppose that $S\subseteq [N]_{\mathcal{O}_{K}}^{d}$ is a set of size $|S|\geq cN^{d-h-1+\varepsilon}$ occupying at most $\alpha \mathcal{N}_{K}(\mathfrak{p})^{d-h}$ residue classes modulo $\mathfrak{p}$ for every prime $\mathfrak{p}\in P$ and some $\alpha>0$. Then if $N$ is sufficiently large there exists a non-zero polynomial $f\in \mathcal{O}_{K}[X_{1},\ldots ,X_{d}]$ of degree $O_{d, h,\varepsilon,\eta, \kappa, K}(1)$ and coefficients of height bounded by $N^{O_{d,h,\varepsilon,\eta,\kappa,K}(1)}$ which vanishes at more than $(1-\eta)|S|$ points of $S$.
\label{theorem 2.4 prima}
\end{theorem}
Now, it is immediate that Theorem \ref{theorem 2.4 prima} implies an adequate version of Theorem \ref{reduction2}, from where we deduce Theorem \ref{basis2}.\\

As it can be seen from the construction of the homogenous polynomial by means of Siegel's lemma, in order to be able to take $r$ large, it is necessary that $h>1$, this leaves untouched the case $k=d-1$ in the statement of Theorem \ref{basis3}. This is because the characteristic subset has size $\lesssim r^{d-h}$ while the dimension of the space of homogeneous polynomials of degree $r$ in $d$ variables is $\sim r^{d-1}$, hence the matrix of the system is of size 
$c_1r^{d-h}\times c_2r^{d-1}$. It is tempting to think that this issue can be avoided if instead of reducing Theorem \ref{basis3} to Theorem \ref{reduction1} one attempts to prove a projective version of Proposition \ref{proposition 2.2}. In this case an obstacle in carrying on the induction is that a badly distributed set of size 
$N^{d-h+\varepsilon}$ in $\mathbb P^{d}(K)$ may have many sections and hence the size of each section is smaller than the desired size. Specifically, in equation \eqref{fact 4.5} of Claim \ref{claim3}, $|\pi_{1}(\overline{S})|$ can be of size $\sim N^2$ instead of size $\sim N^{1+\varepsilon}$, and hence the sections would be of size at least $\gtrsim N^{d-h-2+\varepsilon}$ instead of size at least $\gtrsim N^{d-h-1+\varepsilon}$. If $S$ happens to have small projections, namely that  $|\pi_{j}(\overline{S})|$ is of size $\lesssim N^{1+\varepsilon}$  for all $1\leq j\leq d$, then te same proof would work replacing the larger sieve for a projective version of it, like the one presented in \cite{Zywina}.\\

As a closing remark, we mention that Theorem \ref{basis1}, Theorem \ref{basis2}, Theorem \ref{theorem 2.4}, and Theorem \ref{theorem 2.4 prima} are also valid when 
$S\subseteq [N]_{\mathcal{O}_{K,S}}^{d}$ where, as in section $\S$\ref{section 2}, $\mathcal{O}_{K,S}$ denotes the rings of $S$-integers, for $S$ a finite subset of places containing $M_{K,\infty}$ when $K$ is a number field, and containing the distinguished place $v$ in the case when $K$ is a function field over $\mathbb{F}_q$.

\section{Appendix: effective estimates for heights over global fields}\label{appendix}

In this appendix we provide the proofs of Proposition \ref{serre} and \ref{S-integers bound}. 

In number fields and for $S=M_{K,\infty}$, Proposition \ref{serre} is proved in Chapter 13, Section 13.4  of \cite{Serre}. There, Serre also observed that the constant $c$ can be made effective. In order to prove Proposition \ref{serre}, we adapt Serre's proof to global fields and general rings of $S$-integers. Since all the results in this article are effective, we take the opportunity to fulfill Serre's observation by carrying on the extra work needed to make the dependence of $c$ on $K$ and $S$ explicit. 

\begin{proof}[Proof of Proposition \ref{serre}]
 The statement is trivial if $K=\mathbbm{k}=\mathbb{Q}$ and $S=\{\infty\}$ or $K=\mathbbm{k}=\mathbb{F}_{q}(T)$ and $S$ is the place corresponding to the infinite point in $\mathbb{P}^{1}(K)$: for any set of coordinates of ${\bf x}$, clean denominators and common factors. In this way we obtain coordinates $(x_{0}:\ldots :x_{d})$ such that $x_{i}\in \mathcal{O}_{\mathbbm{k}}$, $\text{gcd}(x_{0},\ldots ,x_{d})=1$ and
$$H_{K}(1:x_{0}:\ldots :x_{d})=H_{K}({\bf x}).$$
For the general case, it will be more convenient to first prove the bound in Proposition \ref{serre} for the absolute height. Let ${\bf x}\in \mathbb{P}^{d}(K)$ and choose coordinates $(x_{0},\ldots ,x_{d})$ with $x_{i}\in \mathcal{O}_{K,S}$. If $\mathfrak{a}_{x_{0},\ldots ,x_{d}}:=\sum_{i=0}^{d}x_{i}\mathcal{O}_{K,S}$ then 
$$H({\bf x})=\dfrac{1}{\mathcal{N}_{K}(\mathfrak{a}_{x_{0},\ldots ,x_{d}})^{\frac{1}{d_{K}}}}\prod_{v\in S}\max_{i}||x_{i}||_{v}.$$
Note that the ideal $\mathfrak{a}_{x_{0},\ldots ,x_{d}}$ depends on the coordinates, but its ideal class depends only on ${\bf x}$. Hence, if we take integral ideals $\mathfrak{a}_{1},\ldots ,\mathfrak{a}_{l}$ representing all the ideal classes of $\mathcal{O}_{K,S}$, satisfying Minkowski's bound $\mathcal{N}_{K}(\mathfrak{a}_{j})\lesssim_{K}1$ for all $j$ (when $K$ is a function field, this bound may be deduced from Section $\S$8.9. in Chapter V of \cite{Lorenzini}), it holds that $\mathfrak{a}_{x_{0},\ldots ,x_{d}}\mathfrak{a}_{j}^{-1}=\alpha \mathcal{O}_{K,S}$ for some $j$ and some $\alpha\in K^{\times}$. Thus, $\alpha^{-1}\mathcal{O}_{K,S}\cdot \mathfrak{a}_{x_{0},\ldots ,x_{d}}=\mathfrak{a}_{j}$. We conclude that $(x_{0}',\ldots ,x_{d}'):=(\alpha^{-1}x_{0},\ldots ,\alpha^{-1}x_{d})$ are coordinates of ${\bf x}$ satisfying $\alpha^{-1}x_{i}\in \mathcal{O}_{K,S}$ for all $i$ and $\mathfrak{a}_{x_{0}',\ldots ,x_{d}'}=\mathfrak{a}_{j}$. In particular,
$$H({\bf x})=\dfrac{1}{\mathcal{N}_{K}(\mathfrak{a}_{j})^{\frac{1}{d_{K}}}}\prod_{v\in S} \max_{i}||x_{i}'||_{v}.$$
Now, we use the following lemma, whose proof will be given after we finish the proof of Proposition \ref{serre}
\begin{lemma}
Let $K$ be a global field and $S$ a set of places which contains the infinite places when $K$ is a number field. There exists a constant $c_{K,S}'>1$, depending only on $K,S$, such that for every ${\bf x}=(x_{v})_{v\in S}\in (\mathbb{R}_{>0})^{|S|}$ there exist $\varepsilon\in \mathcal{O}_{K,S}^{\times}$ and $t>0$ verifying that
\begin{equation}
\left(c_{K,S}'\right)^{-1}\dfrac{x_{v}}{t}\leq ||\varepsilon||_{v}\leq c_{K,S}'\dfrac{x_{v}}{t} \text{ for all }v\in S.
\label{some units}
\end{equation}
\label{some units}
\end{lemma}
Let $c'_{K,S}$ be the constant in Lemma \ref{some units}. Fix any $v_{0}\in S$. Taking $(x_{v})_{v\in S}:=(\max_{i}||x_{i}'||_{v})_{v\in S}$, Lemma \ref{some units} implies that there exist $t>0$ and $\varepsilon\in \mathcal{O}_{K,S}^{\times}$ verifying 
\begin{align}\label{pase2}
\left(c_{K,S}'\right)^{-2}\max_{i}||\varepsilon^{-1}x_{i}||_{v_{0}}&\leq \left(c_{K,S}'\right)^{-1}t \leq \max_{i}||\varepsilon^{-1}x_{i}'||_{v}\\
&\leq c_{K,S}'t 
\leq \left(c_{K,S}'\right)^{2}\max_{i}||\varepsilon^{-1}x_{i}'||_{v_{0}},\nonumber
\end{align}
for all $v\in S$. In particular, if $h:=\max_{i}||\varepsilon^{-1}x_{i}'||_{v_{0}}$, from \eqref{pase2} and the fact that $\prod_{v\in S}||\varepsilon^{-1}||_{v}=1$ we deduce
\begin{equation}
H({\bf x})=\dfrac{1}{\mathcal{N}_{K}(\mathfrak{a}_{j})^{\frac{1}{d_{K}}}}\prod_{v\in S} \max_{i}||\varepsilon^{-1}x_{i}'||_{v}\gtrsim_{K,S} \prod_{v\in S}\max_{i}\{||\varepsilon^{-1}x_{i}'||_{v_{0}}\}\gtrsim_{K,S}h^{|S|}.
\label{pase3}
\end{equation}
Now, note that \eqref{pase2} and \eqref{pase3} imply that for all $v\in S$
$$||\varepsilon^{-1}x_{j}'||_{v}\leq \max_{w}\max_{i}||\varepsilon^{-1}x_{i}'||_{w}\lesssim_{K,S} h\lesssim_{K,S}H({\bf x})^{\frac{1}{|S|}}.$$
Thus,
$$H(1:\varepsilon^{-1}x_{0}':\ldots :\varepsilon^{-1}x_{d}')\lesssim_{K,S} H({\bf x}).$$
Taking $y_{j}:=\varepsilon^{-1}x_{j}'$, we conclude that the coordinates $(y_{0},\ldots ,y_{d})$ satisfy the conclusion of Proposition \ref{serre}. 
\end{proof}
\begin{proof}[Proof of Lemma \ref{some units}]
Note that equation \eqref{some units} is equivalent to prove that there exist $t>0$ and an unit $\varepsilon$ such that
\begin{equation}
\log(t)+\log(||\varepsilon||_{v})=\log(h_{v})+O_{K}(1)\text{ for all }v\in S.
\label{classic problem}
\end{equation}
Denote $(1)_{v\in S}$ for the vector in $\mathbb{R}^{|S|}$ with coordinates all equal to $1$ and let $W$ be the $\mathbb{Z}$-module given by
$$W:=\left\langle \{(\log(||\varepsilon||_{v}))_{v\in S}:\varepsilon\in \mathcal{O}_{K,S}^{\times}\},(1)_{v\in S}\right\rangle_{\mathbb{Z}}\subseteq \mathbb{R}^{|S|}.$$
Denoting by $||\cdot ||$ the $l^{\infty}$-norm in $\mathbb{R}^{|S|}$, we see that in order to prove \eqref{classic problem} it is enough to find a positive constant $C_{W}$ such that for any ${\bf x}\in \mathbb{R}^{|S|}$ there exists ${\bf w}\in W$ satisfying
\begin{equation}
||{\bf x}-{\bf w}||\leq C_{W}.
\label{another equivalent}
\end{equation}
Moreover, we may take $c'_{K,S}=\exp(C_{W})$. In general, it is easy to see that if $W\subseteq \mathbb{R}^{|S|}$ is an additive subgroup satisfying that $\mathbb{R}^{|S|}/W$ is compact, and $\Omega$ is a bounded set containing a representative of each class of $\mathbb{R}^{|S|}/W$, then \eqref{another equivalent} is verified with $C_{W}=\sup_{{\bf y}\in \Omega}||{\bf y}||$. In our case, the subgroup
$$W_{1}:=\left\langle \{(\log(||\varepsilon||_{v})_{v\in S}):\varepsilon\in \mathcal{O}_{K,S}^{\times}\} \right\rangle\subseteq \mathbb{R}^{|S|}$$
defines a lattice of rank $|S|-1$ in the hyperplane $\{(x_{v})_{v\in S}\in \mathbb{R}^{|S|}:\sum_{v\in S}x_{v}=0\}$. Furthermore (see Proposition 5.4.7 (b) and Theorem 5.4.9 (b) in \cite{codes}) it holds that  the volume of $W_{1}$ satisfies the bound:
\begin{equation}
\det(W_{1})\leq \begin{cases}|S|^{\frac{1}{2}}R_{K}h_{K}\displaystyle\prod_{v\in M_{K,\text{fin}}\cap S}\log(\mathcal{N}(\mathfrak{p}_{v}))& \begin{array}{l}\text{if } K\text{ is a number field,}\end{array}\\ |S|^{\frac{1}{2}}\left( 1+q+\dfrac{|X(\mathbb{F}_{q})|-q-1}{g_{K}} \right)^{g_{K}}& \begin{array}{l}\text{if } K \text{is the function field of} \\ \text{a curve } X\text{ over }\mathbb{F}_{q}.\end{array}\end{cases}
\label{bound for the volume}
\end{equation}
In \eqref{bound for the volume}, as usually, $R_{K}$ and $h_{K}$ denote respectively the regulator and the class number of $K$, and $\mathfrak{p}_{v}$ is the prime corresponding to the finite place $v$. Meanwhile, when $K$ is a function field, $g_{K}$ denotes the genus of the curve $X$.

Using \eqref{bound for the volume}, it can be shown that there exists a fundamental system of units $\{\varepsilon_{1},\ldots ,\varepsilon_{|S|-1}\}$ such that
\begin{equation}
\displaystyle \prod_{i=1}^{|S|-1}\log(H(\varepsilon_{i}))\lesssim_{|S|,d_{K}}\det(W_{1}),
\label{Evertse}
\end{equation}
where the implicit constant is effective. A proof of \eqref{Evertse} can be found in Proposition 4.3.9(i) in \cite{Evertse} for number fields, but the same proof carries over to function fields. On the other hand, given  non-zero elements $z_{1},\ldots ,z_{m}\in \mathcal{O}_{K,S}$  which are multiplicatively independent, there exist well known effective lower bounds of the form $\prod_{i=1}^{m}\log (H(z_{i}))\gtrsim_{d_{K},|S|}1$. Indeed, for function fields this is obvious; for number fields, we can use Dobrowolski's theorem (see, for instance, Theorem 4.4.1. in \cite{Bombieri}) or stronger estimates, such as the one given in Corollary 3.1 in \cite{Loher}. Thus, from \eqref{Evertse} we deduce
\begin{equation}
\max_{1\leq i\leq |S|-1}\log(H(\varepsilon_{i}))\lesssim_{|S|,d_{K}}\det(W_{1}).
\label{Evertse2}
\end{equation}
From \eqref{Evertse2} we conclude that we may find a fundamental domain $\Omega_{W_{1}}\subseteq \mathbb{R}^{|S|}$ such that $\max_{{\bf y}\in \Omega_{W_{1}}}||{\bf y}||\lesssim_{|S|,d_{K}}\det(W_{1})$. Moreover, we deduce that there exists a fundamental domain $\Omega_{W}\subseteq \mathbb{R}^{|S|}$ with $\max_{{\bf y}\in \Omega_{W}}||{\bf y}||\lesssim_{|S|,d_{K}}\det(W_{1})$ and thus, the constant $C_{W}$ is effective, and can be taken to be of the form $O_{|S|,d_{K}}(\det(W_{1}))$. Thus we may take $C'_{K,S}=\exp(C_{W})$.
\end{proof}

Now we prove Proposition \ref{S-integers bound} for function fields. Our argument essentially follows the number field case as in Theorem 5.2 in \cite{Lang}, replacing the argument of geometry of numbers by estimates of the dimension of some spaces of divisors. 

\begin{proof}[Proof of Proposition \ref{S-integers bound} for function fields]
Assume that $K$ is a  function field with field of constants $\mathbb{F}_{q}$. We have a map 
\begin{multline*}\varphi:[N]_{\mathcal{O}_{K,S}}\to \{ D\in \text{Div}(K):D\geq 0,\deg(D)\leq \log_{q}(N) \}\\ \times\ \ \{D\in \text{Div}(K):D=\sum_{v\in S}a_{v}\cdot v,|a_{v}|\leq \log_{q}(N)\},
\end{multline*}
$$\varphi(x):=\left( \sum_{v\notin S}\text{ord}_{v}(x)\cdot v,\sum_{v\in S}\text{ord}_{v}(x)\cdot v\right).$$
Note that, modulo constants, the map $x\mapsto \text{div}(x)$ is injective. Thus, $\varphi$ has fibers with $q$ elements. It is clear that
\begin{equation}
\left|\left\{ D\in \text{Div}(K):D=\displaystyle \sum_{v\in S}a_{v}\cdot v:|a_{v}|\leq \log_{q}(N) \right\}\right|\leq \left(\log(N)\right)^{|S|}.
\label{trivial counting}
\end{equation}
Meanwhile, by Lemma 5.6 in \cite{Rosen}, the number of divisor clases of degree zero, $h:=h_{K}$, is finite. By Lemma 5.8 in \cite{Rosen}, for every integer $n$, there are $h$ divisor classes of degree $n$. Suppose that $n\geq 0$ and that $\{\overline{A}_{1},\ldots ,\overline{A}_{h}\}$ are the divisor classes of degree $n$. If $b_{n}$ is the number of effective divisors of degree $n$, then it holds that $b_{n}=\sum_{i=1}^{h}\frac{q^{l(\overline{A}_{i})}-1}{q-1}$, where $l(\overline{A}_{i})$ is the dimension of the Riemann-Roch space associated to $\overline{A}_{i}$. By Exercise 18 in \cite{Rosen}, $l(\overline{A}_{i})\leq \deg(\overline{A}_{i})+1=n+1$, hence
$$b_{n}=\sum_{i=1}^{h}\frac{q^{l(\overline{A}_{i})}-1}{q-1}\leq h\dfrac{q^{n+1}-1}{q-1}\leq 2hq^{n}.$$  
Thus:
\begin{align}\label{norm counting}
\left|\left \{ D\in \text{Div}(K):D\geq 0,\deg(D)\leq \log_{q}(N)\right \}\right|&=\displaystyle \sum_{i\leq \log_{q}(N)}b_{i}\\
&\leq 2h\sum_{i\leq \log_{q}(N)}q^{i}\leq 2hN.\nonumber
\end{align}
By \eqref{trivial counting} and \eqref{norm counting} we conclude
$$|[N]_{\mathcal{O}_{K,S}}|\leq 2qhN(\log(N))^{|S|}.$$
Note that by Proposition 5.11 in \cite{Rosen}, $h\leq (\sqrt{q}+1)^{2g}$, where $g$ is the genus of $K$. Thus, 
$$|[N]_{\mathcal{O}_{K,S}}|\lesssim_{q,g}N(\log(N))^{|S|}.$$
\end{proof}

{\bf Acknowledgments.}  We want to thank the referee for the many comments and suggestions. J. M. Menconi and  M. Paredes were supported in part by CONICET Doctoral Fellowships.  R. Sasyk was partially supported by the grant PIP-CONICET 11220130100073CO. Some of this work was carried out while the second author was a visitor at the Institut Henri Poincar\'e participating in the Trimester Program ``Reinventing Rational Points''. This visit was supported by a CIMPA-Carmin grant. He wishes to thank the organizers, the hospitality of the Institut and the financial support he received there.

\bibliography{paper1}
\bibliographystyle{alpha}

\end{document}